\documentclass{lms}

\usepackage{amsmath}
\usepackage{amssymb}
\usepackage{graphicx}
\usepackage{amsfonts}
\usepackage{mathdots}

\newtheorem{Theorem}{Theorem}[section] 
\newtheorem{Lemma}[Theorem]{Lemma}     
\newtheorem{Corollary}[Theorem]{Corollary}
\newtheorem{Proposition}[Theorem]{Proposition}

\newnumbered{assertion}{Assertion}    
\newnumbered{conjecture}{Conjecture}  
\newnumbered{Definition}{Definition}
\newnumbered{hypothesis}{Hypothesis}
\newnumbered{Remark}{Remark}
\newnumbered{note}{Note}
\newnumbered{observation}{Observation}
\newnumbered{problem}{Problem}
\newnumbered{question}{Question}
\newnumbered{algorithm}{Algorithm}
\newnumbered{Example}{Example}
\newunnumbered{notation}{Notation} 

\newcommand{\ADC}{\operatorname{ADC}}

\def\HH{\mathcal {H}}

\def\R{\mathbb{R}}
\def\D{\mathbb{D}}
\def\T{\mathbb{T}}
\def\N{\mathbb{N}}
\def\Z{\mathbb{Z}}
\def\ov{\overline}

\title{Reverse Carleson embeddings for model spaces} 

\author{Blandign\`{e}res, Fricain, Gaunard, Hartmann and Ross}

\classno{30J05, 30H10, 46E22 (primary)}

\extraline{This work is supported by ANR FRAB: ANR-09-BLAN-0058-02}

\begin{document}
\maketitle

\begin{abstract}
The classical embedding theorem of Carleson deals with finite positive Borel measures $\mu$ on the closed unit disk for which there exists a positive constant $c$ such that $\|f\|_{L^2(\mu)} \leq c \|f\|_{H^2}$ for all  $f \in H^2$, the Hardy space of the unit disk. Lef\`evre et al.\ examined measures $\mu$ for which there exists a positive constant $c$ such that $\|f\|_{L^2(\mu)} \geq c \|f\|_{H^2}$ for all  $f \in H^2$. The first type of inequality above was explored with $H^2$ 
replaced by one of the model spaces $(\Theta H^2)^{\perp}$ by Aleksandrov, Baranov, Cohn, Treil, and Volberg. In this paper we discuss the second type of inequality in $(\Theta H^2)^{\perp}$. 
\end{abstract}



\section{Introduction}

Let $H^2$ be the {classical} Hardy space of the open unit disk $\D$ \cite{Duren,Garnett} 
with norm $\|\cdot\|_2$, and let $H^\infty$ be the space of bounded analytic functions on $\D$. Let $M_{+}(\D)$ denote the finite positive Borel measures on $\D$, and, for $\mu \in M_{+}(\D)$, let $\|\cdot\|_{\mu}$ be the norm in $L^2(\mu)$. 
A beautiful theorem of L. Carleson \cite{Garnett} says that $H^2$  can be continuously embedded into $L^2(\mu)$, i.e.,  
\begin{equation} \label{CE}
\exists  \; c > 0 \; \; \mbox{such that} \; \; \|f\|_{\mu} \leq c \|f\|_2,
 \qquad \forall f \in H^2
\end{equation}
 if and only if 
\begin{equation} \label{Carleson window ineq}
\sup_{I}\frac{\mu(S(I))}{m(I)}<\infty,
\end{equation}
where the above supremum is taken over all arcs $I$ of the unit circle $\mathbb{T} = \partial \mathbb{D}$, $m := d \theta/2 \pi$ is  normalized Lebesgue measure on $\mathbb{T}$, and $S(I)$ is the Carleson window 
\begin{equation} \label{window}
S(I):=\left\{ |z| \leq 1:\;\frac{z}{|z|}\in I,\;1-|z|\leq\frac{m(I)}{2}\right\}.
\end{equation}
We will write $H^2 \hookrightarrow L^2(\mu)$ to represent  the fact that $H^2$ can be continuously embedded into $L^2(\mu)$ via the map 
$f \mapsto f|E_{\mu},$
where $E_{\mu}$ is a carrier of $\mu$.  Measures for which this is true are called \emph{Carleson measures}. 
We will also write statements such as \eqref{CE} in the more convenient form 
$$\|f\|_{\mu} \lesssim \|f\|_2, \quad \forall f \in H^2.$$
Carleson's result can be extended to $\mu \in M_{+}(\D^{-})$ ($\D^{-}$ is the closure of $\D$) but, in the initial definition of embedding in \eqref{CE},  we need to change the phrase $\forall f \in H^2$ into $\forall f \in H^2 \cap C(\D^{-})$, owing to the fact that functions in $H^2$ are defined $m$-almost everywhere on $\T$  while functions in $L^2(\mu)$ are defined $\mu$-almost everywhere on $\T$. However, condition \eqref{Carleson window ineq} clearly implies that the restrictions of Carleson measures to the unit circle are absolutely continuous with respect to $m$, and so the initial concern in examining $\|f\|_{\mu}$ for all $f \in H^2$, and not just the continuous ones, evaporates. 

Lef\`evre et al.\ \cite{Queffelec} examined the \emph{reverse embedding problem}, i.e., when is the above embedding injective with closed range, equivalently, when does \eqref{CE} hold as well as the reverse inequality 
\begin{equation} \label{RCE}
\|f\|_2 \lesssim \|f\|_{\mu}, \quad \forall f \in H^2?
\end{equation}
They proved that if $\mu$ is a Carleson measure, then the reverse embedding happens if and only if \begin{equation} \label{RCW}
\inf_{I} \frac{\mu\left(S\left(I\right)\right)}{m\left(I\right)} > 0.
\end{equation}
In other words, the norms $\|\cdot\|_\mu$ and $\|\cdot\|_2$ are equivalent on $H^2$ if and only if both \eqref{Carleson window ineq} and \eqref{RCW} are satisfied. 

We would like to point out that D. Luecking studied  the question of reverse embeddings for Bergman spaces in \cite{Luecking85,Luecking88} and G. Chac\'on  \cite{chacon} has some related results for certain Dirichlet type spaces. 
\bigskip

The purpose of this paper is to explore reverse embeddings for model spaces $(\Theta H^2)^{\perp} = H^2 \ominus \Theta H^2$, where $\Theta$ is a non-constant 
inner function, that is $\Theta\in H^\infty$ and $\Theta$ admits radial limits of modulus one almost everywhere on $\T$. Aleksandrov \cite{Alek} showed that if $\mu \in M_{+}(\D^{-})$ satisfies
\begin{equation} \label{AE}
\|f\|_{\mu} \lesssim \|f\|_2, \quad \forall f \in (\Theta H^2)^{\perp} \cap C(\D^{-}),
\end{equation}
then each $f \in (\Theta H^2)^{\perp}$ has a finite radial limit at $\mu$-almost every point of $\T$ and the inequality in \eqref{AE} holds for every $f \in (\Theta H^2)^{\perp}$. We point out that, amazingly, $(\Theta H^2)^{\perp} \cap C(\D^{-})$ is dense in $(\Theta H^2)^{\perp}$ \cite{MR1359992} (see also \cite[p.~188]{CRM}).
Again we use the notation $(\Theta H^2)^{\perp} \hookrightarrow L^2(\mu)$ to denote the embedding $f \mapsto f|E_{\mu}$, where $E_{\mu} \subset \mathbb{D}^{-}$ is a carrier of $\mu$.

Treil and Volberg \cite{TV}, along with Cohn \cite{Cohn}, examined when embedding of model spaces actually occurs. In particular, they showed that $(\Theta H^2)^{\perp} \hookrightarrow L^2(\mu)$ as soon as there is an $\varepsilon\in\left(0,1\right)$
such that $\mu$ satisfies the condition \eqref{Carleson window ineq} but the supremum is taken only over arcs $I \subset \T$ satisfying
\begin{equation} \label{eq:intersection window level set}
S\left(I\right)\cap L\left(\Theta,\varepsilon\right)\neq\varnothing,
\end{equation}
where 
\begin{equation} \label{SLS}
L(\Theta, \varepsilon) := \{z \in \D: |\Theta(z)| < \varepsilon\}
\end{equation}
is a sub-level set for $\Theta$. In particular, we see that this condition is weaker than Carleson's one since $(\Theta H^2)^\perp$ functions are much more regular than $H^2$-functions, particularly when we are far from  $L(\Theta,\varepsilon)$. Using weighted Bernstein inequalities, Baranov \cite{Baranov-JFA05} improved the embedding result of Treil--Volberg. 

Conversely, assuming that $L(\Theta, \varepsilon)$ is connected for some $\varepsilon > 0$\footnote{Such inner functions are said to satisfy the \emph{connected level set condition}  (CLS)}, if $(\Theta H^2)^{\perp} \hookrightarrow L^{2}\left(\mu\right)$, then $\mu$ satisfies \eqref{Carleson window ineq}
for arcs $I$ satisfying \eqref{eq:intersection window level set}.

\bigskip

Again, as was asked by Lef\`evre et al.\ for Carleson measures in $H^2$, when is the embedding $(\Theta H^2)^{\perp} \hookrightarrow L^2(\mu)$ injective with closed range (equivalently, induces a reverse embedding)?  For example, if $\mu$ is any one of the \emph{Clark measures} for $\Theta$ (we will define these measures in a moment), then by results of Clark \cite{Clark72} and Poltoratskii \cite{Poltoratskii} we have the isometric embedding $(\Theta H^2)^{\perp} \hookrightarrow L^2(\mu)$. The same is true when $\mu$ is a Clark measure for an inner multiple of $\Theta$.

For another example, let  
$(\lambda_{n})_{n \geq 1} \subset \D$ be a \emph{complete interpolating
sequence} for  $(\Theta H^2)^{\perp}$. This means that 
 the sequence of reproducing kernels $(k^\Theta_{\lambda_n})_{n\geq 1}$ for $(\Theta H^2)^{\perp}$ forms an unconditional basis in $(\Theta
H^2)^{\perp}$ \cite{HNP,Nik2}. 
In this situation it turns out that for the measure 
$$\mu=\sum_{n \geq 1}\frac{1}{\|k_{\lambda_n}\|_2^2}\delta_{\lambda_{n}},$$
the norms $\|\cdot\|_\mu$ and $\|\cdot\|_2$ are equivalent on $(\Theta H^2)^\perp$ and, in particular, we have a reverse embedding. 

For a third example, suppose $M$ is a subspace of $H^2$ satisfying $f/z \in M$ whenever $f \in M$ and $f(0) = 0$. Such subspaces
$M$ are called \emph{nearly invariant} and were initially studied by Hitt and Sarason  \cite{Hitt,Sa88}. 
If $g \in M$ is the unique solution to the extremal problem, 
$$\Re(g(0)) = \sup\{\Re(f(0)) : f \in M, \|f\|_2 \leq 1\},$$ then 
 there exists an inner function $\Theta$ such that $M=g (\Theta H^2)^{\perp}$
and $g$ is an isometric multiplier of $(\Theta H^2)^{\perp}$. That is to say 
$$\|g f\|_2 = \|f\|_2, \quad \forall f \in (\Theta H^2)^{\perp}.$$
Rephrasing this a bit we see that if $d \mu = |g|^2 dm$ then the embedding $(\Theta  H^2)^{\perp} \hookrightarrow L^2(\mu)$ is isometric. As a matter of fact, 
the measures $\mu\in M_+(\T)$ which ensure an isometric embedding of $({\Theta}H^2)^{\perp}$
have been characterized by Aleksandrov \cite{Al98}.
\begin{Theorem}[(Aleksandrov)]\label{ThmAleks}
Let $\mu\in M_+(\T)$. Then the following assertions are equivalent:
\begin{itemize}
\item[(i)] $(\Theta H^2)^{\perp}$ embeds isometrically into $L^2(\mu)$;
\item[(ii)] the function $\Theta$ has non-tangential boundary values $\mu$-almost
everywhere on $\T$ and
\[
 \int_{\T}\left|\frac{1-\overline{\Theta(z)}\Theta(\zeta)}{1-\overline{z}
 \zeta}\right|^2d\mu(\zeta)=\frac{1-|\Theta(z)|^2}{1-|z|^2},
 \quad z\in\D;
\]
\item[(iii)] there exists a $\varphi\in H^{\infty}$ such that $\|\varphi\|_{\infty}\le
1$ and
\begin{eqnarray}\label{ThmAleksClark}
 \int_{\T}\frac{1-|z|^2}{|\zeta-z|^2}d\mu(\zeta)
 =\Re\left(\frac{1+\varphi(z)\Theta(z)}{1-\varphi(z)\Theta(z)}\right),
 \quad z\in\D.
\end{eqnarray}
\end{itemize}
\end{Theorem}
Note that in \cite{deBranges68}, de Branges has proven the result for meromorphic inner functions and in \cite{Krein-Gorbachuck}, Krein has obtained a characterization of  isometric measures for $(\Theta H^2)^\perp$ in a more operator-theoric langage. 
\bigskip 

In the particular case when $\Theta(z)=\Theta_a(z)=\exp(-a\frac{1+z}{1-z})$, $a>0$, the question of equivalence of norms in $(\Theta_aH^2)^\perp$ was discussed in \cite{Kacnelson,Logvinenko,Panejah62,Panejah66}. In this case, the question is equivalent to the following one: for which (positive) measures $\mu$ on the real line
are the norms 
\[
\left(\int_{\mathbb R}|f(x)|^2\,d\mu(x)\right)^{1/2}\quad\hbox{and}\quad \left(\int_{\mathbb R}|f(x)|^2\,dx\right)^{1/2}
\]
equivalent on the space of entire functions of exponential type not exceeding $a/2$ and which belong to $L^2(\mathbb R)$? 

In \cite{Volberg-81} Volberg generalized the previous results and gave a complete answer for general model spaces and absolutely continuous measures $d \mu = w dm, w \in L^{\infty}(\mathbb{T})$.

\begin{Theorem}[(Volberg)]\label{thm:Volberg} Let $d\mu=wdm$, with $w\in L^\infty(\T)$, $w\geq 0$, and let $\Theta$ be an inner function. Then the following assertions are equivalent:
\begin{itemize}
\item[(i)] the norms $\|\cdot\|_{\mu}$ and $\|\cdot\|_2$ are equivalent on $(\Theta H^2)^\perp$;
\item[(ii)] if $(\lambda_n)_{n\geq 1}\subset\D$, then 
\[
\lim_{n\to +\infty}\widehat w(\lambda_n)=0\Longrightarrow \lim_{n\to +\infty}|\Theta(\lambda_n)|=1;
\]
\item[(iii)] we have 
\[
\inf_{\lambda\in\D}(\widehat w(\lambda)+|\Theta(\lambda)|)>0.
\]
\end{itemize}
\end{Theorem}
In the above,  $\widehat w$ represents the harmonic continuation of the function $w$ into the open unit disk, that is,
\[
 \widehat w(z)=\int_{\T}w(\zeta)\frac{1-|z|^2}{|z-\zeta|^2}\,dm(\zeta),\qquad z\in\D.
\]
%

The aim of this paper is to prove a model space version of the Lef\`evre et al.\ reverse Carleson embedding theorem for $H^2$ where we weaken the condition \eqref{RCW} in the spirit of the result of Treil--Volberg for the direct embedding. Along the way, we will also develop a notion of dominating sets for model spaces and use this to  state another reverse embedding theorem. 


\section{Main results}

Our first result is a reverse embedding theorem along the lines of Treil-Volberg for which we need the following notation: 
given an arc $I\subset \T$ and a number $n>0$ we define
the amplified arc  $nI$ as the arc with same center as $I$ and 
of length $n \times m(I)$.

\begin{Theorem}\label{thm:first-reversed-embedding}
Let $\Theta$ be a $(CLS)$ inner function, $L\left(\Theta,\varepsilon_{1}\right)$
its connected sublevel set for a suitable $\varepsilon_1$, 
and $\mu \in M_{+}(\D^{-})$
such that $(\Theta H^2)^{\perp} \hookrightarrow L^{2}\left(\mu\right)$. 
There exists an $N = N(\Theta, \varepsilon_1) > 1$ such that 
if
\begin{equation}\label{eq:reversed-condition-first}
\inf_{I}\frac{\mu(S(I))}{m(I)}>0,
\end{equation}
where the infimum is taken over all arcs $I \subset \T$ with 
$$S\left(NI\right)\cap L(\Theta,\varepsilon_{1})\neq\varnothing,$$
then 
\begin{equation}\label{eq:reversed-equation}
\|f\|_2 \lesssim \|f\|_{\mu}, \qquad \forall f \in (\Theta H^2)^{\perp}.
\end{equation}
\end{Theorem}
The proof of this theorem will show how $N$ depends on 
$\Theta$ and $\varepsilon_1$.
There will also be a discussion in Section \ref{sec:proof-reverse} of why 
the $N$ is needed.

It turns actually out that the $(CLS)$-condition is not needed. Baranov
was able to provide a different proof of this fact after this paper had
been submitted. We will present his proof in a separate appendix.

Our second reverse embedding result involves the notion of a dominating set for $(\Theta H^2)^{\perp}$. A (Lebesgue measurable) subset $\Sigma\subset\mathbb{T}$, with $m\left(\Sigma\right)<1$,
is called a \emph{dominating set} for $(\Theta H^2)^{\perp}$ if 
$$
\int_{\T} |f|^2dm \lesssim \int_{\Sigma}|f|^2 dm, \quad \forall f \in (\Theta H^2)^{\perp}.
$$
Notice how this is equivalent to saying that the measure $d \mu  = \chi_{\Sigma} dm$ satisfies the reverse embedding property for $(\Theta H^2)^{\perp}$. We will discuss 
dominating sets in Section \ref{sec:dominating}. 
It is not too difficult to show that dominating sets always exists for inner functions $\Theta$ such that $\sigma(\Theta)\neq\T$ -- see \eqref{liminf-zero-set} for the definition of $\sigma(\Theta)$. Moreover, if $m(\sigma(\Theta))=0$, then dominating sets can be of arbitrary small Lebesgue measure. This is the case for (CLS) inner functions.  
The situation is more intricate when $\sigma(\Theta)=\T$. We provide
an example of a dominating set in that situation based on Smith-Volterra-Cantor
sets. Kapustin suggested a nice argument showing that dominating sets
always exist. His argument is also presented in Section \ref{sec:dominating}.

With regards to reverse embeddings, we will prove the following result. 

\begin{Theorem}\label{thm:second-reversed-embedding}
Let $\Theta$ be an inner function, $\Sigma$ be a dominating set for
$(\Theta H^2)^{\perp}$, and $\mu \in M_{+}(\D^{-})$ such that $(\Theta H^2)^\perp\hookrightarrow L^2(\mu)$.
Suppose that
\[
\inf_{I}\frac{\mu\left(S\left(I\right)\right)}{m\left(I\right)}>0,
\]
where the above infimum is taken over all arcs $I \subset \T$ such that $$I\cap\Sigma\neq\varnothing.$$
Then
\begin{equation}\label{reversed-inequality}
\|f\|_2 \lesssim \|f\|_{\mu},\qquad \forall f \in (\Theta H^2)^{\perp}.
\end{equation}
\end{Theorem}

In both Theorems \ref{thm:first-reversed-embedding} and
\ref{thm:second-reversed-embedding},  the hypothesis 
$(\Theta H^2)^{\perp}\subset L^2(\mu)$ ensures that the reverse
inequality holds for every function in $(\Theta H^2)^{\perp}$. As
already mentioned, when $\mu$ is carried on $\T$ we might not
be able to define $\|f\|_{\mu}$ properly for certain
$f\in (\Theta H^2)^{\perp}$. Still, it will be clear from the proofs 
that the embeddings \eqref{eq:reversed-equation} and \eqref{reversed-inequality}
still hold on the dense subspaces $(\Theta H^2)^{\perp}
\cap C(\D^-)$ under the reverse Carleson inequality even without the 
assumption on the embedding.

Finally we would like to discuss an alternate proof of Aleksandrov's isometric embedding theorem. Our proof uses the theory of de Branges--Rovnyak spaces. 
\begin{Theorem}[(Aleksandrov)]\label{thm:isometric-Theorem}
Suppose $\Theta$ is an inner function and $\mu\in M_+(\T)$.  Then the embedding $(\Theta H^2)^\perp \hookrightarrow L^2(\mu)$ is isometric if and only if 
there is a function $b$ in the closed unit ball of $H^\infty$ such that 
$$\mu = \sigma_{\Theta b}^1,$$
where $\sigma_{\Theta b}^1$ 
is a so-called Aleksandrov-Clark\footnote{It satisfies \eqref{ThmAleksClark} 
with $\varphi$ replaced by $b$, see also Section
\ref{sec:model-spaces}} measure associated with
${\Theta b}$.
\end{Theorem}

\bigskip
In the following section, we recall some useful facts concerning model spaces $(\Theta H^2)^\perp$ and de Branges--Rovnyak spaces $\HH(b)$. In Section~\ref{sec:isometric-embedding}, we prove Theorem~\ref{thm:isometric-Theorem}. Section~\ref{sec:dominating} is devoted to dominating sets, and in Section~\ref{sec:proof-reverse} we prove Theorems~\ref{thm:first-reversed-embedding} and \ref{thm:second-reversed-embedding}.
In the final section we present Baranov's proof of 
Theorem~\ref{thm:first-reversed-embedding} which does not require the
$(CLS)$-condition.

\section{Model spaces and de Branges--Rovnyak spaces}\label{sec:model-spaces}

Before getting underway, let us first gather up some well-known
facts about the model spaces $(\Theta H^2)^{\perp}$ that will be needed later. References for much of this can be found in \cite{CRM}. Model spaces were initially studied as the typical invariant subspaces for the adjoint of the unilateral shift on $H^2$ \cite{DSS} but the subject has expanded in many ways since then. 

By the Nevenlinna theory \cite{Garnett},  an inner function $\Theta$ can be factored as 
$\Theta = e^{i\varphi}B_{\Lambda} S_{\nu}$, where $\varphi$ is a real constant, $B_{\Lambda}$ is a Blaschke product with zeros (repeated according to multiplicity) $\Lambda = (z_{n})_{n \geq 1}$ and
\[
S_{\nu}(z)=\exp\left(-\int_{\mathbb{T}}\frac{\xi+z}{\xi-z}d\nu(\xi)\right),
\]
where $\nu \in M_{+}(\T)$ with $\nu \perp m$, is a singular inner function. 
The \emph{boundary spectrum} of $\Theta$ is the set
\begin{equation} \label{liminf-zero-set}
\sigma(\Theta):=\left\{ \xi\in\mathbb{T}:\;\varliminf_{z\to\xi}\left|\Theta\left(z\right)\right|=0\right\}.\footnote{We should emphasize that this definition does not take into account
the zeros of $\Theta$ inside $\D$.}
\end{equation}
It turns out that $\sigma(\Theta)$ is equal to 
$(\Lambda^{-}\cap\T)\cup\,\mbox{supt}(\nu)$ -- the union of cluster points of $\Lambda$ on $\T$ together with the support of the measure $\nu$ \cite[p. 63]{Niktr}. 
It is well known  \cite{Mo} that $\Theta$, along with every function in $(\Theta H^2)^{\perp}$, has an analytic continuation 
to an open neighborhood of $\T \setminus \sigma(\Theta)$. 

Several authors \cite{AC,Alek,Cohn} have examined the non-tangential limits of functions in $(\Theta H^2)^{\perp}$ near $\sigma(\Theta)$,  where analytic continuation fails. Indeed, let $ADC_{\Theta}$ denote the set of points $\xi \in \T$ where the angular derivative of $\Theta$, in the sense of Carath\'eodory, exists. More specifically,  $\xi \in ADC_{\Theta}$ if the radial limit of $\Theta$ exists and is unimodular and the radial limit of $\Theta'$ exists. We use the notation $|\Theta'(\xi)|$ to denote the modulus of the angular derivative of $\Theta$ at $\xi$ (when it exists). Note that $|\Theta'(\xi)| > 0$ whenever $\xi \in ADC_{\Theta}$. 
 A
result of Ahern-Clark \cite{AC} yields that $\xi\in ADC_{\Theta}$ if and only
if 
\begin{equation}
|\Theta'(\xi)|=\sum_{n \geq 1}\frac{1-\left|z_{n}\right|^{2}}{\left|\xi-z_{n}\right|^{2}}+2\int_{\mathbb{T}}\frac{d\nu\left(\zeta\right)}{\left|\xi-\zeta\right|^{2}}<\infty.\label{eq:existence derivative}
\end{equation}

The model space $(\Theta H^2)^{\perp}$ is a reproducing kernel Hilbert space with kernel 
\begin{equation}\label{eq:defn-nr}
k_{\lambda}^\Theta(z)=\frac{1-\overline{\Theta(\lambda)}\Theta(z)}{1-\overline{\lambda}z},\quad \lambda\in\D.
\end{equation}
These kernels satisfy the reproducing property 
\begin{equation}\label{eq:defn-reproduisant-property}
\langle f, k_{\lambda}^\Theta \rangle_2 = f(\lambda), \quad \forall f \in (\Theta H^2)^{\perp},
\end{equation}
and have norm 
\begin{equation} \label{k-norm}
\|k_{\lambda}^\Theta\|_2 = \sqrt{\frac{1 - |\Theta(\lambda)|^2}{1 - |\lambda|^2}}.
\end{equation}
It is also possible to define $k^\Theta_{\xi}$ for $\xi \in ADC_{\Theta}$. In this case $k^\Theta_{\xi}\in (\Theta H^2)^{\perp}$, every $f \in (\Theta H^2)^{\perp}$ has a finite non-tangential limit $f(\xi)$ at $\xi$, $\langle f, k_{\xi}^{\Theta} \rangle_2 = f(\xi)$, and $k^\Theta_{\xi}$ has norm
$$\|k^\Theta_{\xi}\|_2 = {\sqrt{|\Theta'(\xi)|}}.$$

For $b$ in the closed unit ball of $H^{\infty}$, the associated \emph{de Branges--Rovnyak space} $\HH(b)$ is defined by 
\[
\HH(b) :=(I-T_bT_b^*)^{1/2}H^2,
\]
and is equipped with the following scalar product
\[
\langle (I-T_bT_b^*)^{1/2}f,(I-T_bT_b^*)^{1/2}g\rangle_b=\langle  f,g\rangle_2,
\]
for any $f,g\in (\hbox{ker}(I-T_bT_b^*)^{1/2})^\perp$, making $\HH(b)$ a Hilbert space contractively contained in $H^2$. Here $T_b$ is the Toeplitz operator with symbol $b$, $T_bf=P_+(bf)$, $f\in H^2$ and $P_+:L^2\to H^2$ is the Riesz projection 
(which is not really needed here since our $b$ is assumed to be analytic). 
When $b=\Theta$ is an inner function,  $\HH(\Theta) = (\Theta H^2)^\perp$ with equality of norms. The space $\HH(b)$ is also a reproducing kernel Hilbert space with kernel $k_\lambda^b$ given by the same formula as \eqref{eq:defn-nr} where $\Theta$ is replaced by $b$. Formulas \eqref{eq:defn-reproduisant-property} and \eqref{k-norm} are also valid in this $\HH(b)$ setting.

For each $\alpha\in\mathbb{T}$, 
$$\Re\left(\frac{\alpha + b(z)}{\alpha - b(z)}\right) > 0, \quad  z \in \D,$$
and hence, by Herglotz's
theorem \cite{Duren}, there is a  unique $\sigma_{b}^{\alpha} \in M_{+}(\T)$, called the \emph{Aleksandrov-Clark measure} \cite{CRM,PolSa,Sa} associated
with $b$ and $\alpha$, such that 
\begin{equation} \label{CMdef}
\Re\left(\frac{\alpha+b(z)}{\alpha-b(z)}\right)=\int_\T \frac{1-|z|^{2}}{|\zeta-z|^{2}}d\sigma_{b}^{\alpha}(\zeta),\quad z\in\mathbb{D}.
\end{equation}
The measure $\sigma_b^\alpha$ is singular (with respect to the Lebesgue measure $m$) if and only if $b$ is an inner function. In this case, $\sigma_{b}^{\alpha}$ is often called a \emph{Clark measure} since they were initially studied by Clark \cite{Clark72}.

Conversely, if we start with a positive Borel measure $\mu$ on $\T$, then we can define the analytic function $b$ by the formula
\begin{equation}\label{eq:clark-Theta-1star}
\frac{1+b(z)}{1-b(z)}=\int_\T \frac{\zeta+z}{\zeta-z}\,d\mu(\zeta),\qquad z\in\D.
\end{equation}
Taking the real part of both sides, we see that the resulting function $b$ is an element in the closed unit ball of $H^\infty$, with $-1<b(0)<1$, and  \eqref{CMdef} holds with $\mu=\sigma_b^1$. The special case $b(0)=0$ happens if and only if $\mu$ is a probability measure. 

Moreover, it follows easily from the uniqueness of the representation \eqref{CMdef} that 
\begin{equation}\label{eq:clark-Theta-1}
\sigma_b^\alpha=\sigma_{\bar\alpha b}^1.
\end{equation}

\bigskip
Clark \cite{Clark72}, for the inner case, and Ball \cite{Ball,BaLu76} for the general case  (see also \cite{Sa}), proved  that for each fixed $\alpha \in \T$ the operator 
\begin{equation}\label{omega-a}
\omega_{b}^{\alpha}: L^{2}(\sigma_{b}^{\alpha}) \to \HH(b), \quad
\omega_{b}^{\alpha}h=\left(1-\overline{\alpha}b\right)C_{\sigma_{b}^{\alpha}}h,
\end{equation}
is an onto partial isometry whose kernel is $(H^2(\sigma_b^\alpha))^\perp$, where 
\[
(C_{\sigma_{b}^{\alpha}}h)(z):=\int_{\mathbb{T}}\frac{h(\xi)}{1-\overline{\xi}z}d\sigma_{b}^{\alpha}(\xi), \quad z \in \D,
\]
denotes the Cauchy transform of the measure $h d \sigma_{b}^{\alpha}$ (see \cite{CRM} for more details) and $H^2(\sigma_b^\alpha)$ is the closure of the polynomials in the $L^2(\sigma_b^\alpha)$-norm.

Poltoratskii \cite{Poltoratskii} (see also \cite{CRM}) went on further to show that for each $h \in L^2(\sigma_{b}^{\alpha})$, 
$$\lim_{r \to 1^{-}} (\omega_{b}^{\alpha} h)(r \xi) = h(\xi), \quad (\sigma_{b}^{\alpha})_s\mbox{-a.e.},$$
where $(\sigma_{b}^{\alpha})_s$ is the singular part of $\sigma_b^\alpha$. In particular, if $b=\Theta$ is an inner function, then $\sigma_\Theta^\alpha$ is singular, $H^2(\sigma_\Theta^\alpha)=L^2(\sigma_\Theta^\alpha)$ and  putting this all together, we have the isometric, in fact unitary, embedding $$(\Theta H^2)^{\perp} \hookrightarrow L^2(\sigma_{\Theta}^{\alpha}),$$i.e., 
\begin{equation} \label{eq:clark identity}
\int_{\mathbb{T}}\left|f\right|^{2}d\sigma_{\Theta}^{\alpha}= \|f\|_2^2, \qquad  \forall f \in (\Theta H^2)^{\perp}.
\end{equation}

\bigskip
For $a\in \D$, let
$$\phi_{a}(z) = \frac{a - z}{1 - \ov{a} z}$$
be the (involutive) conformal automorphism of the
disk which sends $a$ to $0$. Straightforward computations \cite{MR1039571} show that 
\begin{equation}
\sigma_{\phi_{a}\circ b}^{\phi_{a}(\alpha)}=\frac{1}{|\phi_{a}^{'}(\alpha)|}\sigma_{b}^{\alpha}\label{eq:bill's cool formula}.
\end{equation}

If $b=\Theta$ is an inner function, the \emph{Crofoot transform}
\begin{equation}\label{Crofoot}
(U_{\Theta}^{a} f)(z) := \frac{\sqrt{1-|a|^{2}}}{1-\overline{a}\Theta(z)}f(z)
\end{equation}
is a unitary operator from $(\Theta H^2)^{\perp}$ onto $((\phi_{a}\circ\Theta)H^2)^{\perp}$ \cite{Crofoot}.

The Clark measure  $\sigma_{\Theta}^{\alpha}$  is carried by the set 
\begin{equation} \label{carrier}
 E_{\alpha}:=\left\{\xi \in \T: \lim_{r \to 1^{-}} \Theta(r \xi) = \alpha\right\}.
\end{equation}
 By this last statement we mean that the $\sigma_{\Theta}^{\alpha}$-measure of the complement of this set is zero. The carrier is not to be confused with the support of $\sigma_{\Theta}^{\alpha}$ -- which is different. A Clark measure $\sigma_{\Theta}^{\alpha}$ has a point mass at $\xi \in \T$ if and only if $|\Theta'(\xi)| < \infty$ and $\Theta(\xi) = \alpha$. Moreover, in this case $\sigma_{\Theta}^{\alpha}(\{\xi\}) = |\Theta'(\xi)|^{-1}$. Thus if $\sigma_{\Theta}^{\alpha}$ is discrete then 
 $$\sigma_{\Theta}^{\alpha} = \sum_{\{\Theta(\xi) = \alpha, |\Theta'(\xi)| < \infty\}} \frac{1}{|\Theta'(\xi)|} \delta_{\xi}.$$ 
When the carrier of $\sigma_{\Theta}^{\alpha}$ in \eqref{carrier} is discrete, say $(\xi_{n})_{n \geq 1} \subset \T$, Clark \cite{Clark72} showed that 
the system 
\begin{equation} \label{CB}
 \left\{\frac{k^\Theta_{\xi_{n}}}{\sqrt{|\Theta'(\xi_n)|}}: n \geq 1 \right\}
 \end{equation}
  forms an orthonormal
basis for $(\Theta H^2)^{\perp}$ -- called the \emph{Clark basis}. In \cite{fricainJFA} it is shown that this situation cannot occur in the general setting of de Branges--Rovnyak spaces.

An inner function $\Theta$ is said to have the \emph{connected level
set} property (written  $\Theta\in\left(CLS\right)$), if there
exists $\eta\in\left(0,1\right)$ such that the sub-level set 
$$L\left(\Theta,\eta\right):=\left\{ z\in\mathbb{D}:\;\left|\Theta\left(z\right)\right|<\eta\right\} $$
is connected. 


If $\Theta\in(CLS)$, Aleksandrov \cite{Alek} showed that $m(\sigma(\Theta)) = 0$
and moreover, for every $\alpha\in\mathbb{T}$, $\sigma_{\Theta}^{\alpha}(\sigma(\Theta)) = 0$.
Since $\mathbb{T}\setminus\sigma\left(\Theta\right)$ is a countable
union of arcs where $\Theta$ continues analytically, the carrier of 
$\sigma_{\Theta}^{\alpha}$, 
\[
\left\{ \xi\in\mathbb{T}\setminus\sigma\left(\Theta\right):\;\Theta\left(\xi\right)=\alpha\right\} 
\]
is a discrete set $(\xi_{n})_{n \geq 1}$. Hence, as discussed above, every $\Theta \in (CLS)$ has a Clark basis \eqref{CB}.

\section{Isometric Embeddings}\label{sec:isometric-embedding}
In this section, we propose an alternate proof of Aleksandrov's
isometric embedding theorem. The main idea is to first prove the result when $\Theta$ is a Blaschke product with simple zeros and then use an approximation argument based on Frostman shifts. \\

{\emph {Proof of Theorem~\ref{thm:isometric-Theorem}.}} The first implication is Aleksandrov's. We flesh out more of the details. First assume that $\mu=\sigma_{\Theta b}^1$ for some function $b$ in the closed unit  ball of $H^\infty$. Then, by Carath\'eodory's theorem (see \cite[p.6]{Garnett}), there exists a sequence of (finite) Blaschke products $(B_n)_{n\geq 1}$ which converges pointwise to $b$. In particular, if for $z\in\D$, we denote by $$P_z(\zeta)= \frac{1-|z|^2}{|\zeta-z|^2}, \quad \zeta\in\T,$$ the associated Poisson kernel, then according to \eqref{CMdef}, we get that 
for every $z\in\D$,
\[
\int_\T P_z(\zeta)\,d\sigma_{\Theta B_n}^1(\zeta)\to \int_\T P_z(\zeta)\,d\mu(\zeta),\qquad n\to +\infty. 
\]
Since $\|\sigma_{\Theta B_n}^1\|$ is bounded (apply \eqref{CMdef} to $b=\Theta B_{n}$ and $z=0$ and use the fact that $(\Theta B_{n})(0) \to \Theta(0)b(0)$)   and since the closed linear span of $\{P_z:z\in\D\}$ is dense in $C(\T)$, we get that $\sigma_{\Theta B_n}^1\to \mu$ in the weak$-\ast$ topology. In others words, for any $f\in C(\T)$, we have 
\begin{equation}\label{eq:convergence-faible*}
\int_\T f(\zeta)\,d\sigma_{\Theta B_n}^1(\zeta)\to \int_\T f(\zeta)\,d\mu(\zeta),\qquad n\to +\infty.
\end{equation}
Moreover, Clark's theorem says that the embedding $(B_n \Theta H^2)^\perp\hookrightarrow L^2(\sigma_{\Theta B_n}^1)$ is isometric.  Since $(\Theta H^2)^\perp \subset (\Theta B_n H^2)^\perp$, we also have the isometric embedding $(\Theta H^2)^\perp \hookrightarrow L^2(\sigma_{\Theta B_n}^1)$. Thus, according to \eqref{eq:convergence-faible*}, for any $f\in (\Theta H^2)^\perp\cap C(\T)$, we have 
\[
\|f\|_2^2=\int_\T |f(\zeta)|^2\,d\sigma_{\Theta B_n}^1(\zeta)\to \int_\T |f(\zeta)|^2\,d\mu(\zeta).
\]
Then, each function $f\in (\Theta H^2)^\perp$ has a finite radial limit at $\mu$-almost every point of $\T$ and we have $\|f\|_2=\|f\|_\mu$. 

\bigskip
Conversely, suppose that $\mu \in M_{+}(\T)$ and suppose that the embedding $(\Theta H^2)^\perp\hookrightarrow L^2(\mu)$ is isometric. In view of \eqref{eq:clark-Theta-1star}, we know that there exists a function $b_0$  in the closed unit ball of $H^\infty$ such that 
\begin{equation} \label{beezero}
\mu=\sigma_{b_0}^1.
\end{equation} 
 We will now show that $\Theta$ divides $b_0$. We first do this when $\Theta=B$ is a Blaschke product with simple zeros $\Lambda=(\lambda_n)_{n\geq 1}$. 
By our discussion of the Clark theory (generalized by Ball) we have the onto partial isometry
$$\omega_{b_0}: L^2(\mu) \to \HH(b_0), \quad \omega_{b_0} h = (1 - b_0) C_{\mu} h.$$
Note that for any $n\geq 1$, 
 $$k_{\lambda_n}^B(z) =k_{\lambda_n}(z):=\frac{1}{1 - \overline{\lambda_n} z} \in (B H^2)^{\perp} \cap C(\D^{-}),$$
which yields, since $k_{\lambda_n}\in H^2(\mu)=(\hbox{ker }\omega_{b_0})^\perp$, 
\begin{eqnarray}\label{wbkl}
\left\langle \omega_{b_0}k_{\lambda_n},\omega_{b_0}k_{\lambda_\ell}\right\rangle_{b_0}=
\left\langle k_{\lambda_n},k_{\lambda_\ell}\right\rangle_{\mu} =
\left\langle k_{\lambda_n},k_{\lambda_\ell}\right\rangle_2 ,\quad n,\ell\geq 1.
\end{eqnarray}
A standard computation (see for instance \cite[III.6]{Sarason}) 
shows that 
\begin{equation*} 
\omega_{b_0}  k_{\lambda_n}
 =(1-\overline{b_0(\lambda_n)})^{-1}k_{\lambda_n}^{b_0},\qquad n\geq 1.
\end{equation*}
Apply formula \eqref{wbkl} to the above identity and use the reproducing property to get 
\begin{eqnarray*}
\left\langle k_{\lambda_n},k_{\lambda_\ell}\right\rangle_2  & = &(1-\overline{b_0(\lambda_n)})^{-1}(1-b_0(\lambda_\ell))^{-1}\langle k_{\lambda_n}^{b_0},k_{\lambda_\ell}^{b_0}\rangle_{b_0} \\
&=& (1-\overline{b_0(\lambda_n)})^{-1}(1-b_0(\lambda_\ell))^{-1} k_{\lambda_n}^{b_0}(\lambda_\ell)\\
&=&(1-\overline{b_0(\lambda_n)})^{-1}(1-b_0(\lambda_\ell))^{-1} (1-\overline{b_0(\lambda_n)}b_0(\lambda_\ell)) k_{\lambda_n}(\lambda_\ell).
\end{eqnarray*}
Since $$\left\langle k_{\lambda_n},k_{\lambda_\ell}\right\rangle_2=k_{\lambda_n}(\lambda_\ell) = (1 - \overline{\lambda_n} \lambda_{\ell})^{-1} \neq 0,$$ we obtain
\[
(1-\overline{b_0(\lambda_n)})(1-b_0(\lambda_\ell))=1-\overline{b_0(\lambda_n)}b_0(\lambda_\ell),
\]
which, after a little algebra, gives us
\[
\overline{b_0(\lambda_n)}-2\overline{b_0(\lambda_n)}b_0(\lambda_\ell)+b_0(\lambda_\ell)=0.
\]
The above can be re--arranged as 
\[
\overline{b_0(\lambda_n)}(1-b_0(\lambda_\ell))=-b_0(\lambda_\ell)(1-\overline{b_0(\lambda_n)}),\qquad (n,\ell\geq 1).
\]
Setting $f:=b_0/\left(1-b_0\right)$, the last equality implies that,
for every $n\geq 1$, 
\begin{equation}\label{eq:f-lambda}
f(\lambda_n)=-\overline{f(\lambda_1)}=c,
\end{equation}
and so
\begin{equation}\label{eq:defdelta}
b_0(\lambda_n)=\frac{c}{1+c}=:\delta,\qquad n\geq 1.
\end{equation}
Setting $\phi_{\delta}\left(z\right):=\left(\delta-z\right)/\left(1-\overline{\delta}z\right)$,
we see that  $\phi_{\delta}\circ b_0$ vanishes on $\Lambda$ and so
$B$, since it has \emph{simple} zeros,  divides $\phi_{\delta}\circ b_0$. This implies the existence of a function $\vartheta$ in the closed unit ball of $H^\infty$ such that $B\vartheta=\phi_{\delta}\circ b_0$. 
Since $\phi_{\delta} \circ \phi_{\delta}$ is the identity we get $b_0=\phi_{\delta}\circ\left(B\vartheta\right)$,
which shows that 
$$\mu = \sigma_{\phi_{\delta \circ (B \vartheta)}}^{1}.$$ 
To finish this off, we use \eqref{eq:bill's cool formula} and \eqref{eq:clark-Theta-1} to get 
$$\mu = \frac{1}{|\phi_{\delta}'(\phi_{\delta}(1))|} \sigma_{B \vartheta}^{\phi_{\delta}(1)} = \frac{1 - |\delta|^2}{|1 - \delta|^2} \sigma_{B \vartheta}^{\phi_{\delta}(1)}=\frac{1-|\delta|^2}{|1-\delta|^2}\sigma_{B b}^1,$$
where $b$ is the function in the closed unit ball of $H^\infty$ defined by $b=\overline{\phi_\delta(1)}\vartheta$ (note that $|\phi_\delta(1)|=1$). Finally, using \eqref{eq:f-lambda}, we easily see that $c$ is imaginary. Thus
from the definition of $\delta$ in \eqref{eq:defdelta} we have
\[
\frac{1-|\delta|^2}{|1-\delta|^2}=1,
\]
which concludes the proof in the case when $\Theta=B$ is a Blaschke product with simple zeros.  

In the general case, the nice little fact that will contribute here is Frostman's result 
\cite{Frostman-thesis} (see also \cite{MR0231999}) which says that
for an arbitrary inner function $\Theta$,  the \emph{Frostman shift} 
$\phi_{\lambda} \circ \Theta$ is a Blaschke product for every $\lambda \in \D$ with 
the possible exception of a set of logarithmic capacity zero. 
This result can be refined to fit to our situation: $\phi_{\lambda} \circ \Theta$ is a Blaschke product with {\it 
simple} zeros for all $\lambda \in \D$ with the possible exception of a set of Lebesgue area measure zero~
\cite[p.~677]{Fricain}.  So let $(\lambda_n)_{n\geq 1}$ be a sequence in $\mathbb{D}$, $\lambda_n\to 0$, 
such that $B_n:=\phi_{\lambda_n}\circ(-\Theta)$ is a Blaschke product with simple zeros. A trivial estimate shows that 
\begin{equation}\label{eq:Frostman-approximation}
\|B_n-\Theta\|_{\infty}\leq \frac{2|\lambda_n|}{1-|\lambda_n|}\to 0,\quad n\to +\infty.
\end{equation}
Fix $n\geq 1$. Since $\phi_{\lambda_n}\circ \phi_{\lambda_n}$ is the identity, remember from  \eqref{Crofoot} the Crofoot transform $U_{B_n}^{\lambda_n}$ which is a unitary operator from $(B_n H^2)^{\perp}$ onto $(\Theta H^2)^{\perp}$. If we define $\mu'_n \in M_{+}(\T)$ by 
\begin{equation} \label{mu-mu-prime}
d\mu_n':=\frac{1-\left|\lambda_n\right|^{2}}{\left|1-\overline{\lambda_n}B_n\right|^{2}}d\mu,
\end{equation}
then $(B_n H^2)^\perp$ embeds  isometrically into $L^{2}(\mu'_n)$.
Indeed, for $g \in (B_n H^2)^{\perp}$, note that $U_{B_n}^{\lambda_n} g\in (\Theta H^2)^\perp$ and $(\Theta H^2)^\perp\hookrightarrow L^2(\mu)$ isometrically. Thus, 
\[
\int_\T |g|^2 dm=\int_\T |U^{\lambda_n}_{B_n}g|^{2}dm=\int_\T |U^{\lambda_n}_{B_n}g|^{2}d\mu=\int_\T |g|^{2}d\mu'_n.
\]
By the first case of the proof, there is a function $b_n$ in the closed unit ball of $H^{\infty}$ such that $\mu'_n=\sigma_{B_n b_n}^1$, that is
\begin{equation} \label{proper-form}
\mu=\frac{|1-\overline{\lambda_n}B_n|^2}{1-|\lambda_n|^2} \sigma_{B_n b_n}^1,\quad n\geq 1.
\end{equation}
By the Banach-Alaoglu theorem, there is a subsequence $(b_{n_\ell})_{\ell\geq 1}$ and a function $b$ in the closed unit ball of $H^\infty$ such that $b_{n_\ell}$ converges to $b$ in the weak-$*$ topology, which means that 
\[
\int_\T  f b_{n_\ell} \,dm \to \int_\T f b \,dm ,\quad \ell\to +\infty,
\]
for any $f\in L^1(\T)$. In particular, for any $z\in\D$, $b_{n_\ell}(z)\to b(z)$ as $\ell\to +\infty$. Now we argue that $\mu=\sigma_{\Theta b}^1$. Indeed according to \eqref{proper-form}, we have
\begin{eqnarray*}
\int_\T P_z(\zeta)\,d\mu(\zeta)&=&\int_\T P_z(\zeta)\left(\frac{|1-\overline{\lambda_{n_\ell}}B_{n_\ell}(\zeta)|^2}{1-|\lambda_{n_\ell}|^2}-1\right)\,d\sigma_{B_{n_\ell}b_{n_\ell}}^1(\zeta)+\\
&+&\int_\T P_z(\zeta)\,d\sigma_{B_{n_\ell}b_{n_\ell}}^1(\zeta).
\end{eqnarray*}
Observe that 
\[
\frac{|1-\overline{\lambda_n}B_n(\zeta)|^2}{1-|\lambda_n|^2}-1\to 0,\quad n\to +\infty,
\]
uniformly on $\T$ and, according to \eqref{eq:Frostman-approximation}, $B_{n_\ell}$ tends to $\Theta$ uniformly on $\mathbb{D}^{-}$ and $b_{n_\ell}(z)\to b(z)$, $z\in\D$, as $\ell\to +\infty$. Hence,
\[
\int_\T P_z(\zeta)\,d\sigma_{B_{n_\ell}b_{n_\ell}}^1(\zeta)=\frac{1-|B_{n_\ell}(z)b_{n_\ell}(z)|^2}{|1-B_{n_\ell}(z)b_{n_\ell}(z)|^2}
\]
tends to 
\[
\frac{1-|\Theta(z)b(z)|^2}{|1-\Theta(z)b(z)|^2}=\int_\T P_z(\zeta)\,d\sigma_{\Theta b}^1(\zeta).
\]
Thus, we obtain that for any $z\in \D$, 
\[
\int_\T P_z(\zeta)\,d\mu(\zeta)=\int_\T P_z(\zeta)\,d\sigma_{\Theta b}^1(\zeta).
\]
But the closed linear span of $\{P_z:z\in\D\}$ is dense in $C(\T)$, which yields that 
\[
\mu=\sigma_{\Theta b}^1,
\]
and thus concludes the proof. 
\hfill $\square$

\section{Dominating Sets}\label{sec:dominating}

In this section, we introduce and discuss the notion of dominating
sets for $(\Theta H^2)^{\perp}$ where $\Theta$ is a non-constant
inner function. This terminology is perhaps reminiscent of the concept of a dominating sequence for $H^{\infty}$  \cite{BSZ}. 

\begin{Definition}
A (Lebesgue) measurable
subset $\Sigma\subset\mathbb{T}$, with $m\left(\Sigma\right)<1$,
is called a \emph{dominating set} for $(\Theta H^2)^{\perp}$ if 
$$\|f\|_2^{2} \lesssim \int_{\Sigma} |f|^2 dm, \quad \forall f \in (\Theta H^2)^{\perp}.$$
\end{Definition}

Necessarily, in the above definition, $m(\Sigma)>0$.
Also notice that the hypothesis $m(\Sigma) < 1$ is crucial since
otherwise the whole matter becomes trivial. Another observation
is that if $\Sigma$ is a measurable subset of $\T$ such that $m(\Sigma)<1$, then $\Sigma$ is a dominating set if and only if the measure $d \mu = \chi_{\Sigma} dm$ yields a reverse embedding for $(\Theta H^2)^{\perp}$ (and hence equivalent norms).  Therefore, by Volberg's theorem
(Theorem~\ref{thm:Volberg}), 
we obtain a criterion for dominating sets $\Sigma$ in terms of the harmonic continuation of $\chi_\Sigma$ into the open unit disc. However, this criterion is not so easy to deal with. 

One aim of this section is to give simpler necessary or sufficient conditions for dominating sets. More precisely, we would like to highlight some interesting relationship between dominating sets for $(\Theta H^2)^\perp$ and the spectrum of the inner function $\Theta$. Indeed, a simple look at the definition of dominating set tells us that such
a set cannot be too `far' from  the spectrum or from
the points where the modulus of the angular derivative is infinite. We make this more precise below.

It is a remarkable fact that dominating sets always exist. The idea
of the proof of this fact was pointed out to us by V.~Kapustin after
submission of this paper. Before giving his proof at the end of this
section we will present a nice
construction of a dominating set when $\sigma(\Theta)=\T$ based
on Smith-Volterra-Cantor sets.

But first we discuss the relation between dominating sets
and $\sigma(\Theta)$. For this we need some notation: when
$A$, $B$ are sets and $x$ is a point, we set 
\begin{eqnarray*}
d(A, B) &:=& \inf\{|a - b|: a \in A, b \in B\},\\
d(x, A) &:=& d(\{x\},A).
\end{eqnarray*}

\begin{Proposition}
\label{prop dist dom. to ADC}
If $\Sigma$ is a dominating set for
$(\Theta H^2)^{\perp}$, and $\zeta\in \mathbb{T}\setminus ADC_{\Theta}$
then
$d(\zeta,\Sigma)=0$.
\end{Proposition}
\begin{proof}
Let $\zeta \in \mathbb{T}$ such that 
$|\Theta'(\zeta)| = \infty$ and $(\lambda_{n})_{n \geq 1} \subset \mathbb{D}$ with radial limit  $\zeta$.
Let us suppose, towards a contradiction,  that $d(\zeta,\Sigma)>0$. Then, for $n$
large enough, we see that $d(\lambda_{n},\Sigma)\gtrsim1$.
If we consider the reproducing kernels 
$k_{\lambda_{n}}^{\Theta}\in (\Theta H^2)^{\perp}$, the defining property of $\Sigma$ along with \eqref{k-norm}
gives us
\begin{eqnarray*}
\frac{1-|\Theta(\lambda_{n})|^{2}}{1-|\lambda_{n}|^{2}} & = & \int_{\mathbb{T}}|k_{\lambda_{n}}^{\Theta}|^{2}dm
  \lesssim \int_{\Sigma}|k_{\lambda_{n}}^{\Theta}|^{2}dm\\
 & \lesssim & \int_{\Sigma}\left|\frac{\Theta(\xi)-\Theta(\lambda_{n})}{\lambda_{n}-\xi}\right|^{2}dm(\xi)\\
 & \lesssim & \frac{1}{\text{d}(\lambda_{n},\Sigma)^2}\\
 & \lesssim & 1,
\end{eqnarray*}
whereas, by Ahern-Clark \cite{AC}, 
\[
\frac{1-\left|\Theta\left(\lambda_{n}\right)\right|^{2}}{1-\left|\lambda_{n}\right|^{2}}\longrightarrow\infty,\quad n\to\infty,
\]
This yields the required contradiction.
\end{proof}

The following consequence is obvious.

\begin{Corollary}
If $\Sigma$ is a dominating set for $(\Theta H^2)^{\perp}$, then
$$d(\Sigma,\T\setminus\ADC_{\Theta})=0.$$
\end{Corollary}

Now, since $\mathbb{T}\setminus ADC_{\Theta}$ is a subset of the spectrum
$\sigma\left(\Theta\right)$, the previous corollary directly implies
the following one. 
\begin{Corollary}
If $\Sigma$ is a dominating set for $(\Theta H^2)^{\perp}$,  then
$$d(\Sigma,\sigma(\Theta))=0.$$
\end{Corollary}


We will see in Corollary~\ref{thm:dense-dominating-set} 
that if the spectrum of the inner function is the whole circle,
then any dominating set must be dense. However, dense sets
can be of measure zero. Here is a little fact that states that any neighborhood
of a point in the spectrum has to contain portions of the dominating set  $\Sigma$ with 
non negligible Lebesgue measure. 
First let us introduce the \emph{Privalov shadow}: for
$\lambda\in \D$ and $\alpha>0$, let $I_{\lambda}^{\alpha}$ be the
arc in $\T$ centered at $\lambda/|\lambda|$ with length
$\alpha(1-|\lambda|)$.

For $\alpha=1$ we simply write $I_{\lambda}:=I_{\lambda
}^1$, (see Figure \ref{fig:privalov}). Observe
that $I_{\lambda}^{\alpha}$ is the $\alpha$-amplification of $I_{\lambda}$.
\begin{center}
\begin{figure}
\includegraphics[bb=20bp 600bp 400bp 800bp,clip,scale=0.8]{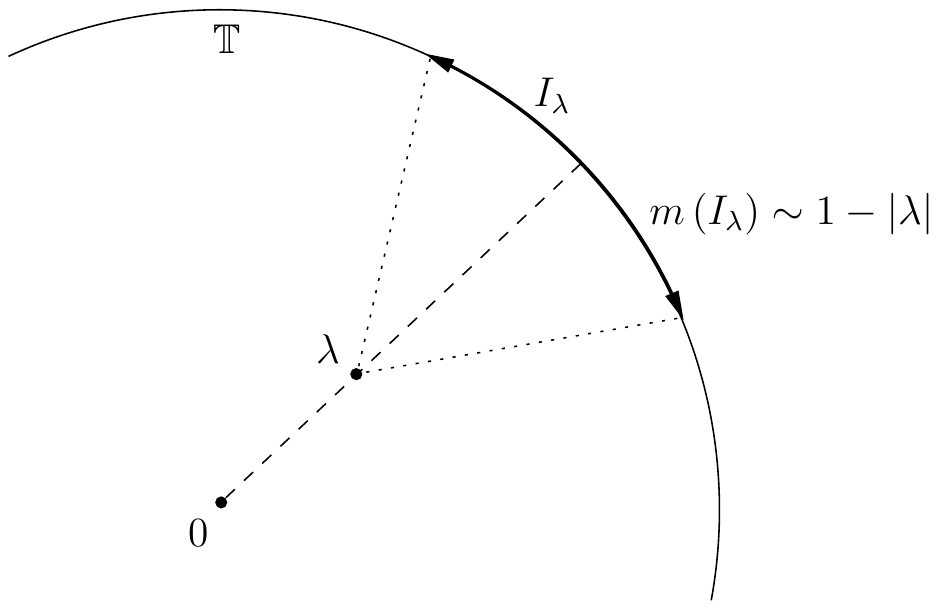}
\caption{\label{fig:privalov}Privalov shadow}
\end{figure}
\end{center}
\begin{Proposition}
Let $\zeta\in\sigma(\Theta)$ and $\Sigma$ dominating. Then there exists an
$\alpha>0$ such that for every sequence $\lambda_n\to \zeta$ with 
$\Theta(\lambda_n)
\to 0$, there is an integer $N$ with
\[
 m(\Sigma\cap I_{\lambda_n}^{\alpha})\gtrsim m(I_{\lambda_n}^{\alpha}),
 \quad n\ge N.
\]
\end{Proposition}

\begin{proof}
Since $\Sigma$ is dominating we have
\[
 c\le \frac{\int_{\Sigma}|k_{\lambda_n}^{\Theta}|^2dm}
 {\|k_{\lambda_n}^{\Theta}\|_2^2}.
\] 
Now, since $\Theta(\lambda_n)\to 0$, there exists an $N$ such that
$|\Theta(\lambda_n)|\le 1/2$ for $n\ge N$, so that in the above inequality
we can replace $k_{\lambda_n}^{\Theta}$ by $k_{\lambda_n}$, $n\ge N$. Hence (with a change of constant from $c$ to $c_1$) for $n\ge N$,
\[
 c_1\le \int_{\Sigma}\frac{1-|\lambda_n|^2}{|\xi-\lambda_n|^2}
 dm(\xi).
\]
Now with an appropriate choice of $\alpha$ we have
\[
 \int_{\T\setminus I_{\lambda_n}^{\alpha}}
 \frac{1-|\lambda_n|^2}{|\xi-\lambda_n|^2}dm(\xi)<\frac{c_1}{2}.
\]
Hence
\begin{eqnarray*}
\lefteqn{ c_1\le \int_{\Sigma}\frac{1-|\lambda_n|^2}{|\xi-\lambda_n|^2} dm(\xi)}\\
 &&=\int_{\Sigma\cap I_{\lambda_n}^{\alpha}}\frac{1-|\lambda_n|^2}{|\xi-\lambda_n|^2}  
  dm(\xi) +
 \int_{\Sigma \cap(\T\setminus I_{\lambda_n}^{\alpha})}\frac{1-|\lambda_n|^2}{|\xi-\lambda_n|^2} dm(\xi)\\
 &&\le \int_{\Sigma\cap I_{\lambda_n}^{\alpha}}\frac{1-|\lambda_n|^2}
 {|\xi-\lambda_n|^2}   dm(\xi) +\frac{c_1}{2}\\
 &&\lesssim\frac{m({\Sigma\cap I_{\lambda_n}^{\alpha}})}{{1-|\lambda_n|^2}}
 +\frac{c_1}{2},
\end{eqnarray*}
which yields the desired conclusion.
\end{proof}

While the general existence result on dominating sets will be
discussed later --- based on Kapustin's ideas --- we include an
existence proof in the easier situation when
the spectrum of the inner function $\Theta$ is not the whole unit circle. 
Here are two results in this direction. 


\begin{Proposition}\label{thm:neighborhood dom.}
Let $\Theta$ be an inner function and assume that $\sigma(\Theta)\neq\T$. Then $(\Theta H^2)^\perp$ has dominating sets. More precisely, if $\Sigma$ is an open subset of $\T$ such that $\sigma(\Theta)\subset\Sigma$ and $m(\Sigma)<1$, then $\Sigma$ is a dominating set for $(\Theta H^2)^\perp$. 
\end{Proposition}

\begin{proof}
Since $\sigma(\Theta)$ is a closed subset of $\T$, and $\sigma(\Theta)\neq\T$, 
we can find an open subset $\Sigma$ of $\T$ such that $\sigma(\Theta)\subset\Sigma$ and $m(\Sigma)<1$.  We now prove that this open set is a dominating set for $(\Theta H^2)^\perp$. As already noticed, this is equivalent to saying if $d \mu = \chi_{\Sigma} dm$ the norms $\|\cdot\|_{\mu}$ and 
$\|\cdot\|_2$ are equivalent on $(\Theta H^2)^\perp$. By Volberg's result we will show that if $\lambda_n\in\D$ is such that $\lim_{n\to +\infty}\widehat{\chi_\Sigma}(\lambda_n)=0$, then $\lim_{n\to +\infty}|\Theta(\lambda_n)|=1$. To do this, we first note that
\begin{equation}\label{eq:harmonic-prolongement}
\widehat{\chi_\Sigma}(\lambda_n)=\int_\Sigma \frac{1-|\lambda_n|^2}{|1-\bar\lambda_n \zeta|^2}\,dm(\zeta).
\end{equation}
Let $\zeta_0\in\T$ be such that some subsequence $\lambda_{n_\ell}\to \zeta_0$. It suffices to show that $\zeta_0\in\T\setminus\sigma(\Theta)$. Indeed 
since we know that the inner function $\Theta$ is analytic on $\T\setminus\sigma(\Theta)$
we will thus get 
\[
\lim_{\ell\to +\infty}|\Theta(\lambda_{n_\ell})|=|\Theta(\zeta_0)|=1.
\]
In order to prove that $\zeta_0\in\T\setminus\sigma(\Theta)$, we argue, assuming to the contrary, that $\zeta_0\in\sigma(\Theta)$. Then for $\zeta\in\T\setminus\Sigma$ and for sufficiently large $\ell$, we have 
\begin{eqnarray*}
|1-\bar\lambda_{n_\ell}\zeta|
 &\geq& |\zeta_0-\zeta|-|\lambda_{n\ell}-\zeta_0|\geq d(\zeta_0,\T\setminus\Sigma)-|\lambda_{n\ell}-\zeta_0|\\
 &\geq& \frac{1}{2} d(\zeta_0,\T\setminus\Sigma).
\end{eqnarray*}
Since $\zeta_0\in\sigma(\Theta)\subset\Sigma$ and $\Sigma$ is open, we have $d(\zeta_0,\T\setminus\Sigma)>0$. 
Hence
\[
\int_{\T\setminus\Sigma}\frac{1-|\lambda_{n_\ell}|^2}{|1-\bar\lambda_n \zeta|^2}\,dm(\zeta)\leq \frac{4(1-|\lambda_{n_\ell}|^2)}{d^2(\zeta_0,\T\setminus\Sigma)}\to 0\qquad\hbox{as }\ell\to+\infty.
\]
On the other hand, by hypothesis, we have $\widehat{\chi_\Sigma}(\lambda_{n_\ell})\to 0$, $\ell\to+\infty$. Taking into account \eqref{eq:harmonic-prolongement}, we get
\begin{eqnarray*}
\lefteqn{\int_\T \frac{1-|\lambda_{n_\ell}|^2}{|1-\bar\lambda_{n_\ell} \zeta|^2}\,dm(\zeta)
 =}\\
  &&\int_\Sigma \frac{1-|\lambda_{n_\ell}|^2}{|1-\bar\lambda_{n_\ell} \zeta|^2}\,dm(\zeta)
 +\int_{\T\setminus\Sigma} 
 \frac{1-|\lambda_{n_\ell}|^2}{|1-\bar\lambda_{n_\ell} \zeta|^2}\,dm(\zeta)
 \to 0
\end{eqnarray*}
as $\ell\to+\infty$.
But this is a contradiction since the left hand side is always $1$.
\end{proof}

\begin{Corollary} \label{thm:arbitrarily small dom. set} Let $\Theta$ be an inner function such that $m(\sigma(\Theta))=0$. Then, for every $0<\varepsilon<1$, there is a dominating set $\Sigma$ for $(\Theta H^2)^\perp$ such that $m(\Sigma)<\varepsilon$.
\end{Corollary}

\begin{proof}
By outer regularity of  Lebesgue measure, we know that for every 
$0<\varepsilon<1$, there exists an open subset $\Sigma$ of $\T$ such that $\sigma(\Theta)\subset\Sigma$ and $m(\Sigma)<\varepsilon$. Now Theorem~\ref{thm:neighborhood dom.} implies that such a $\Sigma$ is a dominating set for $(\Theta H^2)^\perp$.

\end{proof}

\begin{Remark}
Proposition~\ref{thm:neighborhood dom.} and Corollary~\ref{thm:arbitrarily small dom. set}  apply in particular when $\Theta\in (CLS)$, because in that situation we  know by a result of Aleksandrov~\cite{MR1039571} that $m(\sigma(\Theta))=0$.  
\end{Remark}

In the beginning of this section we have seen a simple argument that
a dominating set has to be close to every point of the complement
of $\ADC_{\Theta}$. Also recall that $\sigma(\Theta)$ contains this
complement. 
In order to show that a dominating set has to be close to
every point of $\sigma(\Theta)$ we use Volberg's theorem (note that
$\T\setminus\ADC_{\Theta}$ can be empty while $\sigma(\Theta)$ 
is never empty).

\begin{Proposition}
Let $\Theta$ be an inner function. 
If $\Sigma$ is a dominating set for $(\Theta H^2)^\perp$, then 
for every $\zeta\in\sigma(\Theta)$, we have
$d(\zeta,\Sigma)=0$.  
\end{Proposition}

\begin{proof}
By Volberg's theorem, we know that there exists a $\delta>0$ such that 
\begin{equation}\label{eq:condition-Volberg}
\widehat\chi_\Sigma(z)+|\Theta(z)|\geq \delta,\qquad z\in\D.
\end{equation}
Assume there is $\zeta\in\sigma(\Theta)$
with $d(\zeta,\Sigma)>0$. 
By  \eqref{liminf-zero-set}  there exists a sequence $(z_n)_{n \geq 1}\subset\D$ such that $z_n\to\zeta$ and $\Theta(z_n)\to 0$, as $n\to +\infty$. Now
\[
\widehat\chi_\Sigma(z_n)=\int_\Sigma \frac{1-|z_n|^2}{|z_n-\xi|^2}\,dm(\xi),
\]
and for $\xi\in\Sigma$
\[
|\xi-z_n|\geq |\xi-\zeta|-|\zeta-z_n|\geq d(\zeta,\Sigma)-|\zeta-z_n|.
\]
Since $|\zeta-z_n|\to 0$, as $n\to +\infty$, we get, for $n$ sufficiently large, $|\xi-z_n|\geq d(\zeta,\Sigma^-)/2$, which yields
\[
\widehat \chi_\Sigma(z_n)\leq \frac{4}{d(\zeta,\Sigma^-)^2}(1-|z_n|^2)\to 0,\quad \mbox{as }n\to +\infty,
\]
contradicting  \eqref{eq:condition-Volberg}.   
\end{proof}

We have seen that if the spectrum of $\Theta$ is not the whole circle, then there are many dominating 
sets. Here is a first
result on dominating sets when the spectrum of $\Theta$ is the whole circle.
Then 
the preceding proposition allows to deduce a topological condition on the size of the dominating set. 

\begin{Corollary}\label{thm:dense-dominating-set}
Let $\Theta$ be an inner function such that $\sigma(\Theta)=\T$. If $\Sigma$ is a dominating set for $(\Theta H^2)^\perp$, then $\Sigma$ is dense in $\T$.  
\end{Corollary}

\begin{Remark}
According to Proposition~\ref{thm:neighborhood dom.}, we see that if $\sigma(\Theta)\neq \T$, then one can construct closed dominating sets for $(\Theta H^2)^\perp$. Indeed it is
sufficient to choose an open set $\Sigma$ such that $\sigma(\Theta)\subset\Sigma\subset\Sigma^-$ and $m(\Sigma^-)<1$. Corollary~\ref{thm:dense-dominating-set} shows that  the converse is true. In other words, the space $(\Theta H^2)^\perp$ has {\it closed} dominating sets if and only if $\sigma(\Theta)\neq \T$. In Proposition \ref{DomSetDenseSpectrum}
we will construct a Blaschke product $B$ with $\sigma(B)=\T$ and which admits
an open dominating set.
\end{Remark}

We now discuss, with the help of examples, the conditions
of our previous results and the ``size'' of dominating sets.

\begin{Example} \label{ExamSingInn}
Let $\Theta(z)=\exp(-(1+z)/(1-z)).$ In this case,  $\sigma(\Theta)=\{1\} $
and $\Theta \in (CLS)$. According to Proposition \ref{thm:neighborhood dom.},
every open arc containing $1$ is a dominating set for $\Theta$.
In this example, the {}``open'' condition seems to be necessary.
Indeed, we have the following result.
\end{Example}

\begin{Proposition}
Let $\Theta(z)=\exp(-(1+z)/(1-z)).$ Suppose that $\Sigma\subset \T$ is a closed arc
with endpoints 1 and $e^{i\gamma_0}$, $\gamma_0\in (0,\pi)$,
then $\Sigma$ is not dominating for $(\Theta H^2)^{\perp}$.
\end{Proposition}

\begin{proof}
Let 
\[
 z_n =r_ne^{-i\theta_n},\quad 1-r_n=\theta_n^{3/2},
\]
where $\theta_n>0$ and $\theta_n\longrightarrow 0$. The sequence $z_n$ goes
tangentially ``from below'' to 1, i.e. $\arg(z_n)\in (-\frac{\pi }{2},0)$. In particular we have
\[
 |1-z_n|^2\asymp (1-r_n)^2+\theta_n^2=
  \theta_n^3+\theta_n^2\asymp\theta_n^2,\quad
 n\to\infty.
\]
Hence
\[
 \frac{1-|z_n|^2}{|1-z_n|^2}\asymp
 \frac{1-r_n}{\theta_n^2}=\frac{\theta_n^{3/2}}{\theta_n^2}
 =\frac{1}{\theta_n^{1/2}}\longrightarrow \infty,
\] 
and thus
\[
 |\Theta(z_n)|=\exp\left(-\frac{1-|\lambda_n|^2}{|1-\lambda_n|^2}\right)
 \longrightarrow 0.
\]

Observe that
\begin{eqnarray*}
 \widehat{\chi}_{\Sigma}(z_n)
 &=&\int_{\Sigma}\frac{1-r_n^2}{|e^{it}-z_n|^2}\frac{dt}{2\pi}
 \lesssim\int_{0}^{\gamma_0}\frac{\theta_n^{3/2}}{|e^{it}-e^{-i\theta_n}|^2}dt\\
 &=&\int_{0}^{\gamma_0}\frac{\theta_n^{3/2}}{2(1-\cos(t+\theta_n))}dt
 \asymp\int_{0}^{\gamma_0}\frac{\theta_n^{3/2}}{(t+\theta_n)^2}dt\\
 &=&\theta_n^{3/2}\left[\frac{-1}{t+\theta_n}\right]_0^{\gamma_0}\\
 &\asymp&\theta_n^{1/2}\longrightarrow 0.
\end{eqnarray*}
So $\inf_{\lambda\in\D}(\widehat{\chi}_{\Sigma}(\lambda)+|\Theta(\lambda)|)=0$,
and Volberg's theorem allows us to conclude that $\Sigma$ is not dominating.
\end{proof}

The situation is, of course, the same if we replace $\gamma_0\in (0,\pi)$ by
$\gamma_0\in (-\pi,0)$.

Compare this with the next situation.

\begin{Example}\label{ExamBlProd1}
Let $\Theta=B$ be a Blaschke product with simple zeros $\lambda_n=r_{n}e^{i\theta_{n}}$,
$\theta_{n}=2^{-n}$ and $1-r_{n}=16^{-n}$.
This is an 
inner function such that $\sigma(B)=\{1\}$ and $\ADC_{B}=\mathbb{T}$
(Ahern-Clark). 
We have the following result.
\end{Example}

\begin{Proposition}
Let $B$ be as above.
Suppose that $\Sigma\subset \T$ is a closed arc
with endpoints 1 and $e^{i\gamma_0}$, $\gamma_0\in (0,\pi)$,
then $\Sigma$ is dominating for $(BH^2)^{\perp}$.
\end{Proposition}

\begin{proof}
Again we use Volberg's theorem. Let $(z_n)_{n\geq 1}\subset\D$ be such that $\widehat\chi_{\Sigma}(z_n)\to 0$. We have to show that $|B(z_n)|\to 1$. Pick a convergent subsequence $z_{n_k}\to \zeta\in\T$. Since $\sigma(B)=\{1\}$, the only critical situation is when $\zeta=1$. First we argue that $z_{n_k}\to 1$ "from below". Indeed, assume on the contrary that there is a subsequence, also denoted by $z_{n_k}$, such that $\arg(z_{n_k})\in (0,\pi/2)$. For a point $\lambda\in\D$, recall that $I_\lambda$ denotes the Privalov shadow associated to $\lambda$, that is the arc in $\T$ centered at $\lambda/|\lambda|$ with length $(1-|\lambda|)$ -- see Figure \ref{fig:privalov}. It is easy to see that for $\zeta\in I_\lambda$, we have 
\[
\frac{1-|\lambda|^2}{|\zeta-\lambda|^2}\asymp \frac{1}{1-|\lambda|^2}.
\]
Now, since $z_{n_k}\to 1$, with $\arg(z_{n_k})\in (0,\pi/2)$, there exists an integer $N$ such that for any $k\geq N$, 
we have $I_{z_{n_k}}\cap\T_+\subset\Sigma$ and $m(I_{z_{n_k}}\cap\T_+)\geq (1-|z_{n_k}|)/2$, where $\T_+=\{z\in\T:\arg(z)\in (0,\pi)\}$. Thus 
\begin{eqnarray*}
\widehat\chi_{\Sigma}(z_{n_k})&=&\int_{\Sigma}\frac{1-|z_{n_k}|^2}{|\zeta-z_{n_k}|^2}\,dm(\zeta)\geq \int_{I_{z_{n_k}}\cap\T_+}\frac{1-|z_{n_k}|^2}{|\zeta-z_{n_k}|^2}\,dm(\zeta) \\
&\gtrsim& \frac{1}{1-|z_{n_k}|^2}\frac{1-|z_{n_k}|}{2}\gtrsim 1, 
\end{eqnarray*}
which contradicts the fact that $\widehat\chi_{\Sigma}(z_n)\to 0$. Hence we can assume that $z_{n_k}\to 1$ with $\arg(z_{n_k})\in (-\pi/2,0)$. Taking the logarithmic derivative and using \eqref{eq:existence derivative}, it is easy to see that 
\[
|B'(z)|\leq \sum_{n \geq 1}\frac{1-|\lambda_n|^2}{|1-\bar\lambda_n z|^2},\qquad z\in \D^-,
\]  
and standard estimates show that if $z$ belongs to the closed lower half-disc $\Upsilon:= \lbrace z:\; \vert z\vert \leq 1, \text{Im}(z)\leq 0 \rbrace$, then 
\[
|B'(z)| \lesssim \sum_{n \geq 1}\frac{1}{4^n}<+\infty,
\]
which means that $B'$ is uniformly
bounded on  $\Upsilon$.  Hence $B$ is continuous on $\Upsilon$. In particular, we get that  $|B(z_{n_k})|\to 1$. We conclude from Volberg's theorem that $\Sigma$ is dominating for $(BH^2)^{\perp}$.
\end{proof}

%

We finish this section with an example of a Blaschke product
whose boundary spectrum is the whole circle and which admits a dominating
set $\Sigma$. 

\begin{Proposition}\label{DomSetDenseSpectrum}
There exists a Blaschke product $B$ with $\sigma(B)=\T$
and an open subset
$\Sigma\subsetneq\mathbb{T}$ 
dominating 
for $\left(BH^{2}\right)^{\perp}$.
\end{Proposition}

\begin{proof}

Let $C$ be a Smith-Volterra-Cantor set of $\mathbb{T}$ (i.e., a closed
subset of $\mathbb{T}$, nowhere dense and with positive measure, constructed
in a similar way as the usual Cantor set by removing the middle \emph{fourth} instead of the middle third).
We define $\Sigma:=\mathbb{T}\setminus C$ which is clearly a
dense open subset of $\mathbb{T}$ with $0<m\left(\Sigma\right)<1$.
Since $\Sigma$ is open, there is a sequence of open arcs $\left(I_{n}\right)_{n\geq 1}$
such that $\Sigma=\bigcup_{n \geq 1}I_{n}$. We denote by $\xi_{n}$ the
first endpoint (moving counterclockwise) of $I_{n}$ and $N_{n}$ the integer such that $2^{-\left(2N_{n}+2\right)}\leq \theta_n:=m(I_{n}) <2^{-2N_{n}}$
and $\alpha_n:=\theta_n2^{2N_n}\in [1/4,1)$. Observe that the second 
endpoint of $I_n$ then corresponds to $\xi_ne^{i2\pi\theta_n}$.
We now fix $n$. For each $l> 2N_{n}$ and each $k=1,...,2^{l-2N_{n}}-1$,
we set (see~Figure~\ref{fig:Construction-de})
\[
\lambda_{l,k}^{n}:=\left(1-\frac{1}{2^{2l}}\right)\xi_{n}e^{i\frac{k}{2^{l}}
 2\pi\alpha_n},\;\Lambda_{n}:=\left\{ \lambda_{l,k}^{n}:\; l,k\right\}, \text{ and }\Lambda:=\bigcup_{n \geq 1}\Lambda_{n}.
\]
Observe that every $\lambda_{\ell,k}^n\in\Lambda_{n}$
lies in the Carleson window $S\left(I_{n}\right)$. For each $l\geq N_{n}$,
every arc $\left\{ \left|z\right|=1-2^{-2l}\right\} \cap S\left(I_{n}\right)$
contains $2^{l-2N_{n}}-1$ points of $\Lambda$.

\begin{figure}
\includegraphics[scale=0.8,bb=75bp 400bp 575bp 800bp]{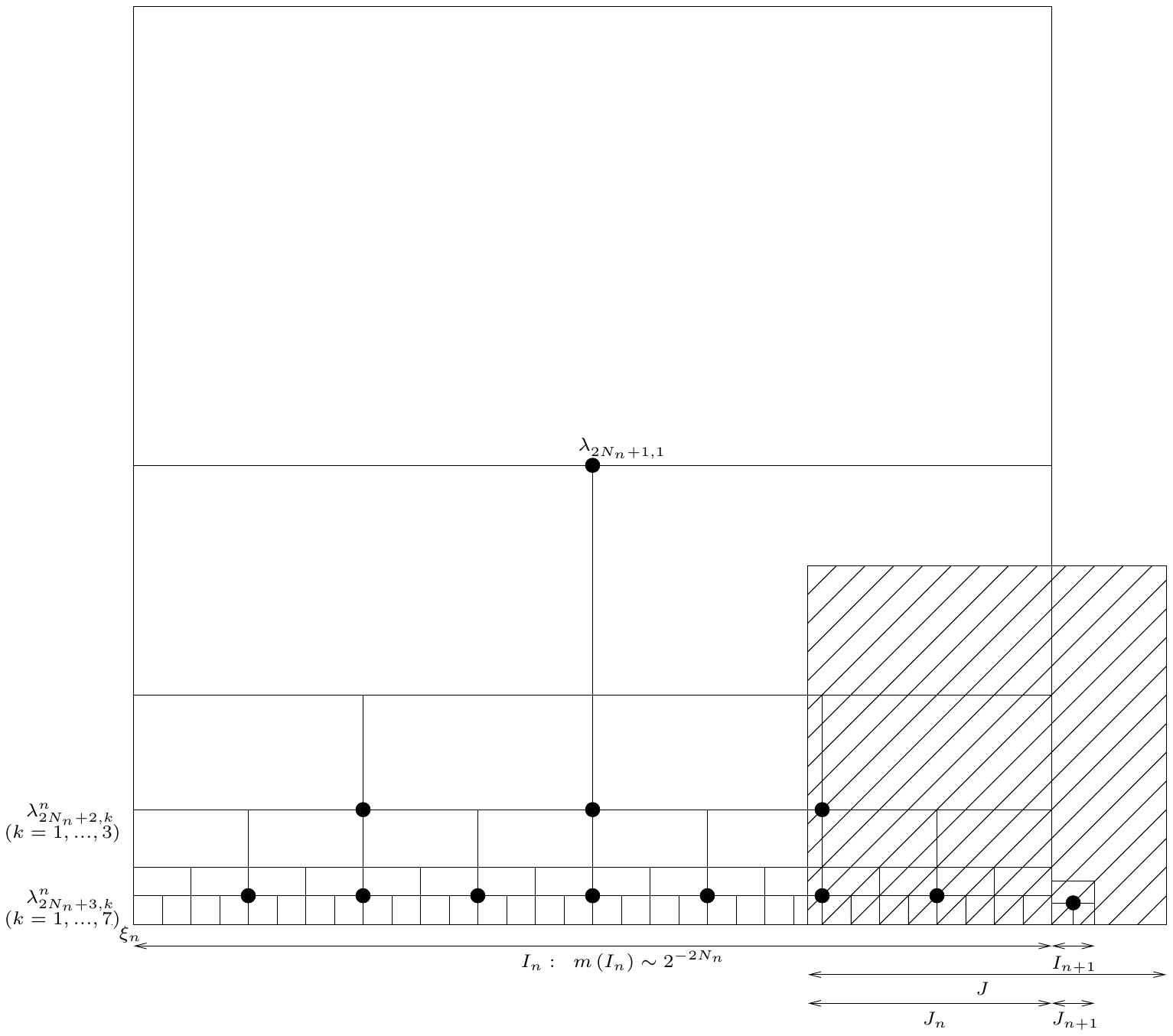}
\caption{\label{fig:Construction-de}$\Lambda_{n}$, $J_{n}$ and $J_{n+1}$}
\end{figure}

Thus, the sequence $\Lambda$ satisfies the Blaschke condition:
\begin{eqnarray*}
\sum_{\lambda\in\Lambda}\left(1-\left|\lambda\right|\right)&=&\sum_{n \geq 1}\sum_{l>2N_{n}}\frac{1}{2^{2l}}\left(2^{l-2N_{n}}-1\right) \\
&\leq & \sum_{n \geq 1}\frac{1}{2^{4N_{n}}}
\\
&\asymp & \sum_{n \geq 1} m(I_n)^2 <\infty. 
\end{eqnarray*}

Let $B$ be the corresponding Blaschke product. It is clear that every
point of $I_{n}$ is a cluster point of $\Lambda$. Since $\Sigma=\bigcup_{n}I_{n}$
is dense in $\mathbb{T}$, it easily follows that $\sigma\left(B\right)=\mathbb{T}$.
It remains to show that $\Sigma$ is a dominating set for $\left(BH^{2}\right)^{\perp}$.
We will use Volberg's theorem and show that 
\[
\inf_{z\in\mathbb{D}}\left(\left|B\left(z\right)\right|+\widehat{\chi_\Sigma}\left(z\right)\right)>0.
\]
The idea is the following. We will show that $\Lambda$ is an interpolating sequence, which implies that
$\left|B\right|$ is big outside a pseudohyperbolic neighborhood  of $\Lambda$  (see
below for definition). Inside a pseudohyperbolic
neighborhood of a point $\lambda\in\Lambda$, an easy computation
will show that $\widehat{\chi_\Sigma}$ is big.

In order to prove that $\Lambda$ is an interpolating sequence, it
suffices to prove that $\nu:=\sum_{\lambda\in\Lambda}\left(1-\left|\lambda\right|^2\right)\delta_{\lambda}$
is a Carleson measure \cite{Garnett} since the sequence $\Lambda$ is separated by construction.
Let $J$ be an arc of $\mathbb{T}$. It is possible to write $J=\bigcup_{n \geq 1}J_{n}$,
where $J_{n}:=J\cap I_{n}$ are disjoint arcs (observe that we have
three configurations: either $J_{n}=\emptyset$ or $I_{n}\subset J$
or $I_{n}$ meets $J$ without being contained in $J$, this latter
situation occurs at most two times at the endpoints of $J$ - see
Figure \ref{fig:Construction-de}).

Let us introduce the sector:
\[
\alpha\left(I\right):=\left\{ z\in\mathbb{D}^{-}:\frac{z}{\left|z\right|}\in I\right\} .
\]
In the following computation, we want to count the number of points
of $\Lambda_{n}$ at level $1-1/2^{2l}$, $l\geq2N_{n}+1$, that fall
in the sector $\alpha\left(K\right)$ where $K$ is any arc in $\mathbb{T}$.
Since the arguments of those points are separated by $2\pi\alpha_n/2^{l}$
(recall that $\alpha_n\in [1/4,1)$) we
get
\[
\#\left(\Lambda_{n}\cap\left\{ \left|z\right|=1-2^{-2l}\right\} \cap\alpha\left(K\right)\right)\le\frac{m\left(K\right)}{2\pi\alpha2^{-l}}
 \asymp2^{l}m\left(K\right).
\]
Hence
\begin{eqnarray*}
\nu\left(S\left(J\right)\right) & = & \sum_{\lambda\in S\left(J\right)}\left(1-\left|\lambda\right|^{2}\right)\lesssim\sum_{\lambda\in S\left(J\right)}\left(1-\left|\lambda\right|\right)\\
 & \leq & \sum_{n \geq 1}\sum_{l>2N_{n}}2^{-2l}\cdot \#\left(\Lambda_{n}\cap\left\{ \left|z\right|=1-2^{-2l}\right\} \cap\alpha\left(J_{n}\right)\right)\\
 & \leq & \sum_{n \geq 1}m\left(J_{n}\right)\sum_{l>2N_{n}}2^{-l}\\
 & \leq & \sum_{n \geq 1}m\left(J_{n}\right)\\
 & \leq & m\left(J\right),
\end{eqnarray*}
so that $\nu$ is a Carleson measure and thus $\Lambda$ is an interpolating
sequence. So, for arbitrarily fixed $\eta\in\left(0,1\right)$, if
\[
z\in\mathbb{D}\setminus\bigcup_{\lambda\in\Lambda}\Omega\left(\lambda,\eta\right),\text{ with }\Omega\left(\lambda,\eta\right):=\left\{ z\in\mathbb{D}:\:\left|\frac{\lambda-z}{1-\overline{\lambda}z}\right|<\eta\right\} ,
\]
then $\left|B\left(z\right)\right|\asymp1$ (see for instance \cite[p. 218]{Niktr}). On the other hand, if
$z\in\Omega\left(\lambda,\eta\right)$, then
\begin{eqnarray*}
\widehat{\chi_\Sigma}\left(z\right) & \asymp & \int_{\Sigma}\frac{1-\left|\lambda\right|^{2}}{\left|\lambda-\xi\right|^{2}}dm\left(\xi\right)\\
 & \geq & \int_{I_{\lambda}\cap\Sigma}\frac{1-\left|\lambda\right|^{2}}{\left|\lambda-\xi\right|^{2}}dm\left(\xi\right)\\
 & \gtrsim & \frac{m\left(I_{\lambda}\cap\Sigma\right)}{1-\left|\lambda\right|},
\end{eqnarray*}
where $I_{\lambda}$ is the Privalov shadow of $\lambda$.  Let $n$ be such that $\lambda\in\Lambda_{n}$.  Then $\lambda\in S\left(I_{n}\right)$
and thus $m\left(I_{\lambda}\cap\Sigma\right)\geq m\left(I_{\lambda}\cap I_{n}\right)\geq m\left(I_{\lambda}\right)/2\asymp1-\left|\lambda\right|,$
so that $\widehat{\chi_\Sigma}\left(z\right)\gtrsim1$, $z\in\Omega\left(\lambda,\eta\right)$.
Finally, we obtain that 
\[
\inf_{z\in\mathbb{D}}\left(\left|B\left(z\right)\right|+\widehat{\chi_\Sigma}\left(z\right)\right)>0,
\]
which ends the proof.
\end{proof}

We have seen that if $\sigma(\Theta)\neq\T$ then $(\Theta H^2)^\perp$ admits a dominating set. In Proposition~\ref{DomSetDenseSpectrum}, we have constructed an example of inner functions $\Theta$ such that $\sigma(\Theta)=\T$ and the corresponding model space admits also a dominating set. It is thus
natural to ask whether dominating sets always
exits. The answer to this question is affirmative. Indeed, during the 2012 conference held in St Petersburg, V. Kapustin 
suggested an idea for the proof of this fact based on the Aleksandrov
disintegration formula (see \cite{CRM}). 
\\


\begin{Theorem}[(Kapustin)]
Every model space admits a dominating sets.
\end{Theorem}

\begin{proof}
First recall the Aleksandrov disintegration formula:
for $f\in L^1$, we have 
\[
\int_\T f(\zeta)\,dm(\zeta)=\int_\T\left(\int_\T f(\zeta)\,d\sigma_\Theta^\alpha(\zeta)\right)\,dm(\alpha).
\]
Pick any partition of $\T$, $\T=A_1\cup A_2$ with $m(A_i)\in (0,1)$ and set 
\[
T_i=\Theta^{-1}(A_i)=\{\zeta\in\T:\Theta(\zeta)\in A_i\}.
\]
It follows from the disintegration formula 
\[
\int_{T_i}|f(\zeta)|^2\,dm(\zeta)=\int_\T\left(\int_\T \chi_{T_i}(\zeta)|f(\zeta)|^2\,d\sigma_\Theta^\alpha(\zeta)\right)\,dm(\alpha).
\]
Recall now that $E_\alpha=\{\zeta\in\T:\Theta(\zeta)=\alpha\}$ is a carrier for $\sigma_\Theta^\alpha$, see \eqref{carrier}. 
Since $E_\alpha\subset\T\setminus T_i$ when $\alpha\in\T\setminus A_i$, 
we have 
\[
\int_\T \chi_{T_i}(\zeta)|f(\zeta)|^2 \,d\sigma_\Theta^\alpha(\zeta)=0.
\]
For similar reason, if $\alpha\in A_i$, since $E_\alpha\subset T_i$, then 
\[
\int_\T \chi_{T_i}(\zeta)|f(\zeta)|^2 \,d\sigma_\Theta^\alpha(\zeta)=\int_\T |f(\zeta)|^2 \,d\sigma_\Theta^\alpha(\zeta).
\]
Using Clark's isometric embedding theorem, we finally obtain that 
\[
\int_{T_i}|f(\zeta)|^2\,dm(\zeta)=\int_{A_i}\|f\|_2^2\,dm(\alpha)=m(A_i)\|f\|_2^2. 
\]
It remains to check that $m(T_i)<1$. For this, observe that the previous equality implies in particular that $m(T_i)$ is nonzero as soon as $m(A_i)>0$. Since this is true for $A_i$ and its complementary, we conclude that $m(T_i)<1$ and so $T_i$ is a dominating set. 
\end{proof}


\section{Reverse Embeddings for $(\Theta H^2)^\perp$ -- proofs of Theorem~\ref{thm:first-reversed-embedding} and Theorem~\ref{thm:second-reversed-embedding}}\label{sec:proof-reverse}

The proof of Theorem~\ref{thm:first-reversed-embedding} requires a few preliminaries on perturbation of bases. Recall that a sequence $(x_n)_{n \geq 1} \subset \mathcal{H}$ is a \emph{Riesz basis} for a separable Hilbert space $\mathcal{H}$  if the closed linear span of $(x_{n})_{n \geq 1}$ is $\mathcal{H}$ and 
$$
\|\sum_{n \geq 1} a_n x_n\|_{\mathcal{H}}^2  \asymp \sum_{n \geq 1} |a_n|^2, \quad \forall (a_{n})_{n \geq 1} \in \ell^2(\mathbb{N}).
$$ 
Also recall that if $\Theta \in (CLS)$, then a carrier of the Clark measure $\sigma_{\Theta}^{\alpha}$,
$$\{\xi \in \T \setminus \sigma(\Theta): \Theta(\xi) = \alpha\},$$ 
is a discrete set $(\xi_n)_{n \geq 1}$ and so 
$\{k^{\Theta}_{\xi_n}/\|k^{\Theta}_{\xi_n}\|_2: n \geq 1\}$ is an orthonormal basis, 
a so-called Clark basis, for $(\Theta H^2)^{\perp}$. 
Recall that $\|k^{\Theta}_{\xi_n}\|_2=\sqrt{|\Theta'(\xi_n)|}$.

The following
result is 
due to Baranov 
(\cite[Corollary 1.3 and proof of Theorem 1.1]{Baranov-II}).

\begin{Theorem}[(Baranov)] \label{thm:perturbation Baranov Dyakonov}Let
$\Theta \in (CLS)$. There exists $\varepsilon_0=\varepsilon_0(\Theta)\in (0,1)$
making the following true: 
if $(k^\Theta_{\xi_{n}}/\|k^\Theta_{\xi_n}\|_2)_{n \geq 1}$, with $\xi_n \in ADC_\Theta$, 
is a Riesz basis for $(\Theta H^2)^{\perp}$ and  $\lambda_{n} \in \D^{-}$ satisfy 
\[
|\lambda_{n}-\xi_{n}|<\varepsilon_0 |\Theta'(\xi_{n})|^{-1},
\]
then $(k^\Theta_{\lambda_n}/\|k^\Theta_{\lambda_n}\|_2)_{n \geq 1} $
is also a Riesz basis for $(\Theta H^2)^{\perp}$. Moreover, there is a positive constant
$C$ such that for every $f \in (\Theta H^2)^{\perp}$,
we have
\begin{equation} \label{eq:perturbation}
\sum_{n\geq1}\frac{|f(\xi_{n})-f(\lambda_{n})|^{2}}{|\Theta'(\xi_{n})|}\leq\varepsilon_0 C\|f\|_2^2.
\end{equation}
\end{Theorem}

In particular one can choose 
$(k^\Theta_{\xi_{n}}/\|k^\Theta_{\xi_n}\|_2)_{n \geq 1}$ to be a Clark basis. Let us mention that Cohn \cite{Cohn86} also established an interesting result about the stability of Clark bases for one-component inner functions.

\begin{proof}[of Theorem~\ref{thm:first-reversed-embedding}]
Let 
$\varepsilon_0\in\left(0,1\right)$ be the constant given in Theorem \ref{thm:perturbation Baranov Dyakonov} and let $(k^\Theta_{\xi_n}/\|k^\Theta_{\xi_n}\|_2)_{n \geq 1}$ be a Clark basis (we know that such a basis exists because $\Theta\in (CLS)$). For $\eta=\varepsilon_0 \sqrt{2}/2$, we define $\sigma_{n}\subset \T$ 
to be  the arc centered at $\xi_{n}$ of length 
$\eta|\Theta'(\xi_{n})|^{-1}$. Since $\eta<\varepsilon_0$, we easily check that 
$$S(\sigma_{n})\subset D(\xi_{n},\varepsilon_0 |\Theta'(\xi_{n})|^{-1}).$$
Recall  from  \eqref{window} that $S(\sigma_n)$ is the Carleson window over $\sigma_n$.
Moreover, we argue that 
\begin{equation}\label{eq:deux-point-bases}
|\xi_n-\xi_m|\geq \varepsilon_0 \min(|\Theta'(\xi_n)|^{-1}, |\Theta'(\xi_m)|^{-1}),\qquad n\neq m. 
\end{equation}
Indeed, for a point $\xi_n$, if for some $m\neq n$, there exists a point $\xi_m\in D(\xi_n,\varepsilon_0 |\Theta'(\xi_n)|^{-1})$, then we could define
\[
\lambda_\ell=\begin{cases}
\xi_\ell&\hbox{for }\ell\neq m\\
\xi_n&\hbox{for }\ell=m.
\end{cases}
\]
While this sequence satisfies the perturbation condition of Theorem~\ref{thm:perturbation Baranov Dyakonov},
it is certainly no longer a basis (the same vector appears two times). 

Now it follows from \eqref{eq:deux-point-bases} that $\sigma_n\cap \sigma_m=\varnothing$, $n\neq m$, and thus the Carleson windows $S(\sigma_n)$, $n\geq 1$, are disjoint. Moreover, since $\Theta\in (CLS)$, we know \cite{Baranov-II} that
\[
\min\left(|\Theta'(\xi_{n})|^{-1}, d(\xi_{n},\sigma(\Theta))\right )
 \asymp  d(\xi_{n},L(\Theta,\varepsilon_{1})),
\]
where $L(\Theta,\varepsilon_1)$ is the connected sub-level set for $\Theta$ defined in \eqref{SLS}. In particular,
\[
|\Theta'(\xi_{n})|^{-1}
 \ge c d(\xi_{n},L(\Theta,\varepsilon_{1}))
\]
for a suitable constant $c$.
The definition of $\sigma_{n}$ yields that 
\[
m(\sigma_{n})\geq \eta  c  d(\xi_{n},L(\Theta,\varepsilon_{1})),
\]
or more explicitely,
\[
  d(\xi_{n},L(\Theta,\varepsilon_{1})) \le \frac{m(\sigma_n)}{\eta c}
 =\frac{|\Theta'(\xi_n)|^{-1}}{c},
\]
meaning that $D(\xi_n,(c|\Theta'(\xi_n)|)^{-1})$ meets $L(\Theta,\varepsilon_1)$
(see Figure 3).
Since $D(\xi_n,(c|\Theta'(\xi_n)|)^{-1})\subset S(\frac{2}{\eta c}\sigma_n)$
it is enough to 
pick $N=2/(\eta c)$. So
\[
S(N\sigma_{n})\cap L(\Theta,\varepsilon_{1})\neq\varnothing,
\]
which implies, by hypothesis, that for every $n$
\begin{equation}\label{eq:reversed-S-sigma_n}
\mu (S(\sigma_{n}))\geq c_{N} m(\sigma_{n}).
\end{equation}
\begin{figure}
\includegraphics[scale=0.3]{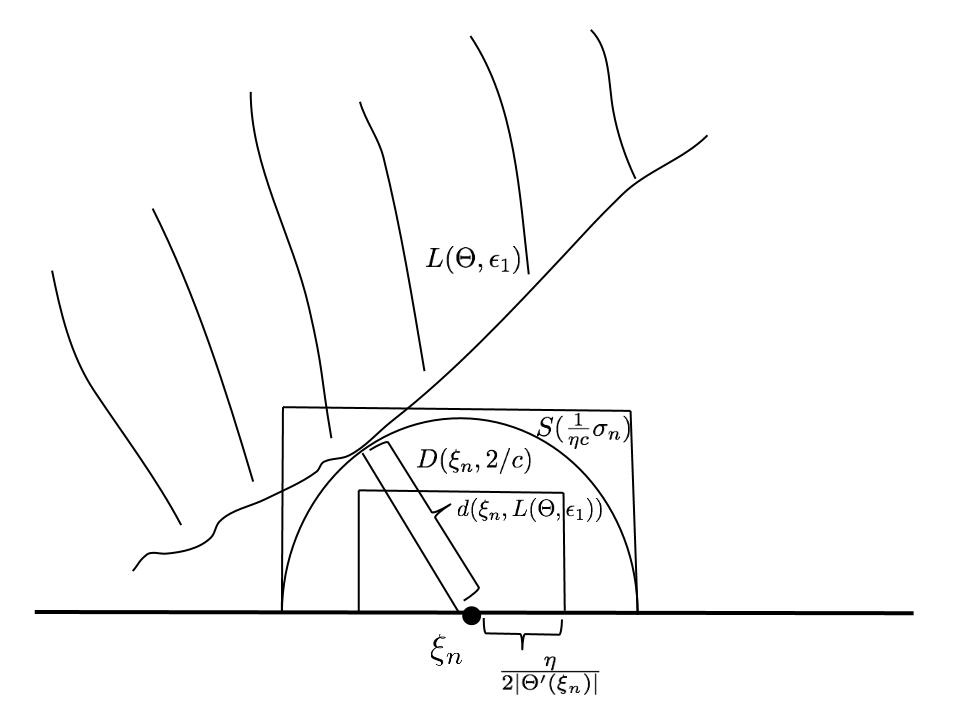}
\caption{}
\end{figure}

Let us now take $f\in (\Theta H^2)^{\perp}$. Since, as mentioned in the introduction, $(\Theta H^2)^\perp\cap C(\D^{-})$ is dense in $(\Theta H^2)^\perp$ and since $(\Theta H^2)^\perp$ embeds continuously into $L^2(\mu)$, we may assume that $f$ is continuous on $\D^-$. Define $\mu_{n}$ as the point of $S(\sigma_{n})$ satisfying
\[
|f(\mu_{n})|=\inf_{\zeta\in\sigma_{n}}|f(\zeta)|.
\]
Since $(k^\Theta_{\xi_n}/\sqrt{|\Theta'(\xi_n)|})_{n \geq 1}$ is an orthonormal basis of $(\Theta H^2)^\perp$, we have
\[
\|f\|_2^2 =\sum_{n \geq 1}\frac{|f(\xi_{n})|^{2}}{|\Theta'(\xi_{n})|}\leq 2\left(\sum_{n \geq 1}\frac{|f(\xi_{n})-f(\mu_{n})|^{2}}{|\Theta'(\xi_{n})|}+\sum_{n \geq 1} \frac{|f(\mu_{n})|^{2}}{|\Theta'(\xi_{n})|} \right).
\]
Using Theorem~\ref{thm:perturbation Baranov Dyakonov}, 
along with \eqref{eq:reversed-S-sigma_n}, and remembering that the Carleson windows $S(\sigma_n)$ are disjoint, we get   
\begin{eqnarray}\label{estimate-p.21}
\|f\|_2^2  & \leq & 2\varepsilon_0 C \|f\|_2^2+\frac{2}{\eta}\sum_{n\geq 1} m(\sigma_{n})|f(\mu_{n})|^{2}\nonumber\\
 & \leq & 2\varepsilon_0 C \|f\|_2^2+\frac{2}{\eta c_{N}}\sum_{n \geq 1}\mu(S(\sigma_{n}))|f(\mu_{n})|^{2}\nonumber\\
 & \leq & 2\varepsilon_0 C \|f\|_2^2+\frac{2}{\eta c_{N}}\sum_{n \geq 1}\int_{S(\sigma_{n})}|f|^{2}d\mu\nonumber\\
 & \leq & 2\varepsilon_0 C \|f\|_2^2+\frac{2}{\eta c_{N}}\int_{\overline{\mathbb{D}}}\left|f\right|^{2}d\mu.
\end{eqnarray}
Finally, for sufficiently small $\varepsilon_0$  we obtain 
\[
\|f\|_2^2 \leq\frac{2}{(1-2\varepsilon_0C)\eta c_{N}}\int_{\overline{\mathbb{D}}}\left|f\right|^{2}d\mu,
\]
which yields the required estimate. 
\end{proof}

\begin{Remark}
Note that from the first line of \eqref{estimate-p.21} we obtain
\[
 \|f\|_2^2\le\frac{2}{(1-2\varepsilon_0C)\eta}\int_{\Sigma}
 |f|^2dm,
\]
where $\Sigma=\bigcup \sigma_n$. This gives another way of constructing
dominating sets (at least when the Clark measure is discrete): here we take neighborhoods of the support points of the
Clark basis instead of an open neighborhood of the spectrum as in 
Theorem \ref{thm:neighborhood dom.}.
\end{Remark}

\begin{Remark}
Here is a simple example showing that it is not sufficient
that condition \eqref{eq:reversed-condition-first} of 
Theorem~\ref{thm:first-reversed-embedding} 
is satisfied  if $N$ is not suitably chosen. Let us discuss this example
in the upper half plane and consider
the special case of $\Theta\left(z\right)=e^{2iz}$.
It is known that with this $\Theta$, the model space $(\Theta H^2)^{\perp}$
is isomorphic to the Paley-Wiener space $PW_{\pi}^{2}$ (the space
of entire functions of exponential type at most $\pi$, whose restriction
to $\mathbb{R}$ belongs to $L^{2}$). 
Given $\varepsilon>0$, the sub-level set is
exactly
\[
L\left(\Theta;\varepsilon\right)=\left\{ z\in\mathbb{C}:\;\text{Im}\left(z\right)>\theta:= \frac{1}{2}\ln\frac{1}{\varepsilon} \right\} .
\]
In what follows we will choose $\varepsilon=e^{-6}$ so that $\theta=3$.

Note that for the sets $(\sigma_n)_n$ appearing in the proof of Theorem~\ref{thm:first-reversed-embedding}, by the Kadets-Ingham theorem we can take $\sigma_n=(n-\delta,n+\delta)$,
$n\in\Z$, where $0<\delta<1/4$ (see for instance \cite[Theorem D4.1.2]{Nik2}).

Let us now 
consider $\Lambda=\left\{ \lambda_{n};\; n\in\mathbb{Z}\setminus\left\{ 0\right\} \right\} $
with $\lambda_{n}:=n+\frac{1}{8} i$ and 
\[
\mu_\Lambda:=\sum_{n\geq 1}\delta_{\lambda_{n}}.
\]
Note that $\Z$ is a complete interpolating sequence for $PW^2_{\pi}$ (see e.g. \cite{Seip04}) as will be $\Lambda\cup \{\lambda_0\}$.
Hence, we can find $f\in PW_{\pi}^{2}$, such that $f\left(\lambda_{n}\right)=0$,
$n\in\mathbb{Z}\setminus\left\{ 0\right\} $
and $f(\lambda_0)=1$. In particular, $\int\left|f\right|^{2}d\mu_\Lambda=0$
and of course $\left\Vert f\right\Vert _{2}\neq0$. So, the reverse
embedding fails, while \eqref{eq:reversed-condition-first} is valid
for $N=1$. Indeed, if $I$ is an interval such that $S\left(I\right)\cap L\left(\Theta,\varepsilon\right)\neq\varnothing$, 
then $m(I)\ge 3$ and since $0\le \lambda_{n+1}-\lambda_n\le 2$ we have
$S(I)\cap \Lambda\neq \varnothing$.
%
It actually turns out that
$\mu_\Lambda\left(S\left(I\right)\right)=\#\left(S\left(I\right)\cap\Lambda\right)
\asymp m(I)$.

An appropriate choice for $N$ here is $N\ge 6$. In particular
$S(6\sigma_0)$ meets $L(\Theta,\varepsilon)$. But $S(\sigma_0)$ does not contain
any point of $\Lambda$ so that $\mu(S(\sigma_0))=0\not\gtrsim m(\sigma_0)$
and \eqref{eq:reversed-condition-first} fails as it should.
\\

The above example is also instructive in that it allows to observe that 
Theorem~\ref{thm:first-reversed-embedding} 
does not apply when the sequence $\Lambda$ is ``too far'' from $\R$. Consider
for instance $\Lambda_1=\{n+i;n\in\Z\setminus\{0\}\}$ and 
$\Lambda_2=\{n+i;n\in\Z\}$. For $\Lambda_2$ we get reverse embedding
while for $\Lambda_1$ we won't.
Observe that for both sequences we will have
$\mu(S(\sigma_n))=0$ for every $n$, so that the reverse Carleson
inequality fails.
\end{Remark}

Now we give the proof of our reverse embeddings result which involves dominating sets. We will need the following
lemma, for which we omit the proof.
\begin{Lemma}
\label{lem:partition dominating set}Let $\Sigma\subsetneq\mathbb{T}$
and $\delta>0$. It is possible to find a (finite) sequence $(\sigma_{n})_{n \geq 1}$
of disjoint semi-open arcs of $\mathbb{T}$ such that $\Sigma\subset\bigcup_{n}\sigma_{n}$,
$\Sigma\cap\sigma_{n}\neq\emptyset$ and $m(\sigma_{n})<\delta$.
\end{Lemma}

\begin{proof}[of Theorem~\ref{thm:second-reversed-embedding}]
Since $(\Theta H^2)^\perp\cap C(\D^{-})$ is dense in $(\Theta H^2)^\perp$ and since $(\Theta H^2)^\perp \hookrightarrow L^2(\mu)$, it suffices
to show the inequality $\|f\|_2\lesssim \|f\|_{\mu}$, for every function $f\in (\Theta H^2)^\perp$  continuous on $\D^-$. Fix $\varepsilon>0$.
Since $f$ is uniformly
continuous, it is possible to find a $\delta>0$ such
that, for every arc $I$ of length less than $\delta$, we have 
\[
|f(z)-f(\mu)|<\varepsilon, \quad \forall z,\mu\in S(I). 
\]
According to Lemma \ref{lem:partition dominating set}, we can construct
a sequence of disjoint semi-open arcs of length less than $\delta$, covering
and intersecting $\Sigma$. In particular, $m(\sigma_n) \leq k \mu(S(\sigma_n))$ for some $k > 0$ independent of $n$.
We now introduce the points $\xi_{n}$ and $z_{n}$ such that 
\[
|f(\xi_{n})|=\max\{|f(\xi)|: \xi\in\sigma_{n}^{-}\},\quad |f(z_{n})|=\min\{|f(z)|:  z\in S(\sigma_{n}^{-})\}
\]
These points satisfy
\[
|f(\xi_{n})|\leq\varepsilon+|f(z_{n})|.
\]

Since $\Sigma$ is dominating, there is some constant $c > 0$ such that 
\begin{align*}
\|f\|_2^2 & \leq c \int_{\Sigma} |f|^2 dm\\
& \leq c \sum_{n \geq 1} m(\sigma_n) |f(\xi_n)|^2\\
& \leq 2 c \left( \varepsilon^2 \sum_{n \geq 1} m(\sigma_n) + \sum_{n \geq 1} m(\sigma_n) |f(z_n)|^2\right)\\
& \leq 2c \left( \varepsilon^2  + k \sum_{n \geq 1} \mu(S(\sigma_n)) |f(z_n)|^2\right)\\
& \leq 2c  \varepsilon^2  + 2 c k \sum_{n \geq 1} \int_{S(\sigma_n)} |f|^2 d \mu\\
& \leq 2 c  \varepsilon^2 +2 c k \int_{\D^{-}} |f|^2 d \mu.
\end{align*}\
Since $\varepsilon$ can be arbitrarily small, we have the desired estimate $\|f\|_2 \lesssim \|f\|_{\mu}$.
\end{proof}

\begin{Remark}
If $\Theta$ is inner and $\sigma(\Theta)=\T$, then, according
to Corollary \ref{thm:dense-dominating-set}, any dominating set $\Sigma$
for $(\Theta H^2)^{\perp}$ is dense in $\T$. Hence requiring the reverse
Carleson inequality $\mu(S(I))\gtrsim m(I)$ for any arc $I\subset \T$
meeting $\Sigma$ is the same as requiring this inequality for {\it every}
arc $I\subset \T$. This situation is completely described by
Lef\`evre et al. Hence Theorem
\ref{thm:second-reversed-embedding} is interesting for inner functions $\Theta$
for which $\sigma(\Theta)\not =\T$. In that case Theorem 
\ref{thm:neighborhood dom.} states that
dominating sets always exist.
\end{Remark}

\section{Baranov's proof}

After submission of our paper, Anton Baranov pointed out that in Theorem~\ref{thm:first-reversed-embedding} the assumption that $\Theta\in (CLS)$ is not essential. With his kind permission, we include his proof of this result which 
is different in flavor and  which is based on the Bernstein-type 
inequalities in model spaces he obtained in \cite{Baranov-JFA05,Baranov-09}.
It uses a Whitney type decomposition of $\T\setminus \sigma(\Theta)$. 
Let $\varepsilon>0$, let $\delta\in (0,1/2)$ and let 
\[
d_\varepsilon(\zeta)=d(\zeta,L(\Theta,\varepsilon)),
\]
where we recall that $L(\Theta,\varepsilon)=\{z\in\D:|\Theta(z)|<\varepsilon\}$. Since 
\[
\int_{\T\setminus\sigma(\Theta)}d_\varepsilon^{-1}(\zeta)\,dm(\zeta)=\infty,
\]
we can choose a sequence of arcs $I_k$ with pairwise disjoint interiors such that $\bigcup_k I_k=\T\setminus\sigma(\Theta)$ and 
\[
\int_{I_k}d_\varepsilon^{-1}(\zeta)\,dm(\zeta)=\delta.
\]
In this case\footnote{Note that such a system of arcs was also considered in \cite{Baranov-09} for $\delta=1/2$.}
\begin{equation}\label{eq:distance-level-set}
\frac{1-\delta}{\delta}m(I_k)\leq d(I_k,L(\Theta,\varepsilon))\leq\frac{1}{\delta}m(I_k).
\end{equation}
Indeed, by the definition of $I_k$, there exists $\zeta_k\in I_k$ such that $d_\varepsilon(\zeta_k)=\frac{1}{\delta}m(I_k)$, whence for any $\zeta\in I_k$, we have 
\[
d_\varepsilon(\zeta)\geq d_\varepsilon(\zeta_k)-m(I_k)\geq \frac{1-\delta}{\delta}m(I_k).
\]
It follows from \eqref{eq:distance-level-set} that 
\[
m(I_k)\int_{I_k}d_\varepsilon^{-2}(u)\,dm(u)\leq \left(\frac{\delta}{1-\delta}\right)^2.
\]
Now recall the definition of the weight involved in
the Bernstein-type inequality 
\[
w_p(z)=\|k_z^\Theta\|_q^{-\frac{p}{p+1}},
\]
where $1\leq p<\infty$ and $q$ is the conjugate exponent of $p$. Later on we will choose 
$p$ such that $1\leq p<2$. Then it is shown in \cite[Lemmas 4.5 \& 4.9]{Baranov-JFA05} that 
\[
w_p(\zeta)\geq C_0 d_\varepsilon(\zeta),
\]
where $C_0$ depends only on $p$ and $\varepsilon$ (but not on $\Theta$). Thus
\begin{equation}\label{eq:weight-integral}
m(I_k)\int_{I_k} w_p^{-2}(\zeta)\,dm(\zeta)\leq C_1(p,\varepsilon) \delta^2.
\end{equation}

Let $I_k^{(j)}$, $j=1, \dots 4$ be the quarters of $I_k$
and let $S_k^{(j)}$ be the parts of $S_k$ lying over  $I_k^{(j)}$. Thus, $S_k = \bigcup_{j=1}^4 S_k^{(j)}$
(note that $S_k^{(j)}$ are not standard Carleson windows). By \eqref{eq:distance-level-set}, we have 
\[
S(NI_k^{(j)})\cap L(\Theta,\varepsilon)\neq \emptyset
\]
as soon as $N>\frac{8}{\delta}$. 
This will be the choice of $N$ in the Theorem.
Suppose now that
\[
 A:=\inf_I\frac{\mu(S(I))}{m(I)}>0,
\]
where the infimum is taken over all arcs $I\subset\T$ with
$S(NI)\cap L(\Theta,\varepsilon_1)\neq\emptyset$.
 Then we have 
\[
\mu(S_k^{(j)})\geq \mu(S(I_k^{(j)}))\geq A m(I_k^{(j)}).
\]

Now let $f\in K_\Theta$ be continuous in $\D\cup\T$. By the mean value property, there exists $s_k^{(j)}\in S_k^{(j)}$ such that 
\begin{equation}\label{eq:int-S-k-j}
\int_{S_k^{(j)}} |f|^2d\mu = |f(s_k^{(j)})|^2 \mu(S_k^{(j)})  \geq A m(I_k^{(j)}) \cdot |f(s_k^{(j)})|^2.
\end{equation}
Denote by 
\[
\mathfrak{J}_k^{i,j}=\int_{I_k^{(i)}}|f(u)-f(s_k^{(j)})|^2\,dm(u).
\]
Then we have
\begin{eqnarray*}
\sum_k \int_{I_k}|f|^2\,dm&=&\sum_k \left(\int_{I_k^{(1)}}|f(u)|^2+\int_{I_k^{(2)}}|f(u)|^2+\int_{I_k^{(3)}}|f(u)|^2+\int_{I_k^{(4)}}|f(u)|^2 \right)\,dm(u) \\
&\leq&2\sum_k (\mathfrak{J}_k^{1,3}+\mathfrak{J}_k^{2,4}+\mathfrak{J}_k^{3,1}+\mathfrak{J}_k^{4,2})\\
&+&2\sum_k\left(|f(s_k^{(3)})|^2 m(I_k^{(1)})+|f(s_k^{(4)})|^2 m(I_k^{(2)})+|f(s_k^{(1)})|^2 m(I_k^{(3)})+|f(s_k^{(2)})|^2 m(I_k^{(4)}) \right).
\end{eqnarray*}
Since $m(I_k^{(1)})=m(I_k^{(2)})=m(I_k^{(3)})=m(I_k^{(4)})$, we get with \eqref{eq:int-S-k-j}
\[
\sum_k \int_{I_k}|f|^2\,dm\leq 2\sum_k (\mathfrak{J}_k^{1,3}+\mathfrak{J}_k^{2,4}+\mathfrak{J}_k^{3,1}+\mathfrak{J}_k^{4,2})+2A^{-1}\|f\|_{\mu}^2.
\]

Let us now estimate $\sum_k \mathfrak{J}_k^{1,3}$. We have 
\[
\mathfrak{J}_k^{1,3}=\int_{I_k^{(1)}}|f(u)-f(s_k^{(3)})|^2\,dm(u)=\int_{I_k^{(1)}}\left|\int_{[s_k^{(3)},u]}f'(v)\,|dv|\right|^2\,dm(u),
\]
where $[s_k^{(3)},u]$ denotes the interval with endpoints $s_k^{(3)}$ and $u$ and $|dv|$ stands for the Lebesgue measure on this interval. Using Cauchy-Schwarz' inequality, we obtain 
\[
\mathfrak{J}_k^{1,3}\leq \int_{I_k^{(1)}}\left(\int_{[s_k^{(3)},u]}|f'(v)|^2w^2_p(v)\,|dv|\right) \left( \int_{[s_k^{(3)},u]}w_p^{-2}(v)\,|dv|  \right)\,dm(u).
\] 
Now recall that the norms of reproducing kernels in model spaces have a certain monotonicity along the radii. More precisely, let $q>1$. Then it is shown in \cite[Corollary 4.7.]{Baranov-JFA05} that there exists $C=C(q)$ such that for any $z=\rho e^{it}$ and $\tilde z=\tilde\rho e^{it}$ with $0\leq \tilde\rho\leq \rho$, we have
\begin{equation}\label{norm-rep-kern-monotone}
\|k_{\tilde z}^\Theta\|_q\leq C(q) \|k_z^\Theta\|_q.
\end{equation}
Using \eqref{norm-rep-kern-monotone}, \eqref{eq:weight-integral} and the fact that the angle\footnote{That explains why we choose a decomposition with $\mathfrak{J}_k^{i,j}$, $i\neq j$, since in this case the interval $[s_k^{(j)},u]$, $u\in I_k^{(i)}$, will never be orthogonal to the boundary.} between $[s_k^{(3)},u]$ and $\T$ is separated from $\frac{\pi}{2}$, we conclude that 
\[
m(I_k)\int_{[s_k^{(3)},u]} w_p^{-2}(v)\,|dv|\leq C_1(p) m(I_k) \int_{I_k} w_p^{-2}(v)\,|dv|\leq C_2(p,\varepsilon) \delta^2.
\]
Hence
\[
\sum_k\mathfrak{J}_k^{1,3}\leq C_2(p,\varepsilon) \delta^2 \sum_k \frac{1}{m(I_k)}\int_{I_k^{(1)}} \int_{[s_k^{(3)},u]}|f'(v)|^2 w_p^2(v)\,|dv|\,dm(u).
\]
Again just by the mean value property, there exists $u_k\in I_k^{(1)}$ such that 
\[
\sum_k \frac{1}{m(I_k)}\int_{I_k^{(1)}}\int_{[s_k^{(3)},u]}|f'(v)|^2 w_p^2(v)\,|dv|\,dm(u)=\frac{1}{4}\sum_k\int_{[s_k^{(3)},u_k]}|f'(v)|^2 w_p^2(v)\,|dv|.
\]
Now note that the measure $\sum_k m_{[s_k^{(3)},u_k]}$ (sum of Lebesgue measures on the intervals) 
is a Carleson measure with a uniform bound on the Carleson constant independent of the location of $u_k\in I_k^{(1)}$ and $s_k^{(3)}\in S_k^{(3)}$. Then by the Bernstein's inequality, we have 
\[
\sum_k\int_{[s_k^{(3)},u_k]}|f'(v)|^2 w_p^2(v)\,|dv|\leq C_2(p)\|f\|_2^2,
\] 
which gives 
\[
\sum_k \mathfrak{J}_k^{1,3}\leq C_3(p,\varepsilon) \delta^2 \|f\|_2^2.
\]
Using similar estimates for the other terms $\sum_k \mathfrak{J}_k^{2,4}$, $\sum_k \mathfrak{J}_k^{3,1}$ and $\sum_k \mathfrak{J}_k^{4,2}$, we obtain 
\[
\sum_k \int_{I_k}|f|^2\,dm\leq C_4(p,\varepsilon) \delta^2 \|f\|_2^2+2A^{-1}\|f\|_\mu^2.
\]
Finally note that for the integrals over $\sigma(\Theta)=\T\setminus\bigcup_k I_k$, we have 
\begin{equation}\label{eq:integral-over-spectrum}
\int_{\sigma(\Theta)}|f|^2\,dm\leq \tilde A \|f\|_\mu,
\end{equation}
with $\tilde A$ depending only on $A$. Indeed for any $\rho>0$, there exists arcs $J_\ell$ with pairwise disjoint interiors such that $|J_\ell|\leq \rho$, $\sigma(\Theta)\subset \bigcup_\ell J_\ell$ and $S(J_\ell)\cap L(\Theta,\varepsilon)\neq\emptyset$. Then \eqref{eq:integral-over-spectrum} follows in a trivial way from the proof in \cite{Queffelec}. Thus we obtain finally 
\[
\|f\|_2^2 \leq (\tilde{A}+2A^{-1})\|f\|_\mu^2+C_4(p,\varepsilon) \delta^2 \|f\|_2^2,
\]
that is 
\[
(1-C_4(p,\varepsilon)\delta^2)\|f\|_2^2\leq (\tilde{A}+2A)\|f\|_\mu^2.
\]
It remains to choose $\delta$ so small that $C_4(p,\varepsilon)\delta^2<1$. Note that $\delta$ (and thus $N$) depend only on $\varepsilon$ and some fixed $1\leq p<2$. 
\bigskip

\begin{acknowledgements}\label{ackref}
We would like to warmly thank A. Baranov for his comments and 
suggestions, and in particular for his proof of Theorem~\ref{thm:first-reversed-embedding} which does not require the $(CLS)$-condition. We also thank V. Kapustin for the proof he
suggested for the existence of dominating sets. 
\end{acknowledgements}

\bibliographystyle{plain}

\bibliography{references}

\affiliationone{Alain Blandig\`eres and Emmanuel Fricain\\
   Institut Camille Jordan, Universit\'e Claude Bernard Lyon 1,  69622 Villeurbanne C\'edex\\
   France
   \email{blandigneres@math.univ-lyon1.fr\\
   fricain@math.univ-lyon1.fr}}
\affiliationtwo{Fr\'{e}d\'{e}ric Gaunard and Andreas Hartmann\\
  Institut de Math\'ematiques de Bordeaux, Universit\'e Bordeaux 1, 351 cours de la Lib\'eration 33405 Talence C\'edex \\
   France
   \email{Frederic.Gaunard@math.u-bordeaux1.fr\\
   Andreas.Hartmann@math.u-bordeaux1.fr}}
\affiliationthree{William T. Ross\\
  Department of Mathematics and Computer Science, University of Richmond, Richmond, VA 23173 \\
  USA
   \email{wross@richmond.edu}}

\end{document}